\documentclass{report}

\newtheorem{theorem}{Theorem}
\newtheorem{lemma}{Lemma}
\newtheorem{proof}{Proof}

\newtheorem{remark}{Remark}
\newtheorem{example}{Example}

\usepackage{color}
\usepackage{mathrsfs}
\usepackage{upgreek}
\usepackage{todonotes}

\newcommand{\mblue}{\color{black}}
\newcommand{\ourdesi}{\flat}

\usepackage{amsfonts,amsmath,amssymb}

\usepackage{yfonts}

\usepackage{siunitx}
\title{On the turnpike phenomenon for optimal
boundary control problems with hyperbolic systems}

\author{Martin Gugat, Falk M. Hante
\thanks{
Friedrich--Alexander--Universit\"at Erlangen--N\"urnberg,
Department Mathematik,
Cauerstr.~11, 91058 Erlangen,
Germany (\texttt{martin.gugat@fau.de, falk.hante@fau.de})}
}

\date{\today}

\DeclareSIUnit{\length}{\meter}
\DeclareSIUnit{\time}{\second}
\DeclareSIUnit{\speed}{\meter\per\second}
\DeclareSIUnit{\density}{\kilogram\per\cubic\meter}
\DeclareSIUnit{\pressure}{\kilogram\per\meter\per\square\second}
\DeclareSIUnit{\flux}{\kilogram\per\square\meter\per\second}
\DeclareSIUnit{\roughness}{}
\DeclareSIUnit{\friction}{\per\meter}
\DeclareSIUnit{\gravity}{\meter\per\square\second}

\begin{document}

\maketitle


{\bf Abstract:}
%
We study
problems of
optimal boundary control with systems governed by
linear hyperbolic partial differential equations.
The objective function is quadratic
and given by an integral
over the finite time interval $(0,\, T)$
that depends on the boundary traces of the solution.
If the time horizon $T$ is sufficiently large,
the solution of
the dynamic  optimal boundary control problem
can be approximated by the solution of a steady state optimization problem.
We show that  for $T\rightarrow\infty$  the approximation error converges
to zero in the sense of the norm in
 $L^2(0,\,1)$ with the rate $1/T$,
if the time interval $(0,\, T)$
is transformed to the fixed interval $(0,\, 1)$.
%
Moreover, we show that also for
optimal boundary control problems with
integer constraints for the controls
the  turnpike phenomenon
occurs.
In this case the steady state optimization problem
also has the integer constraints.
If $T$ is sufficiently large, the integer part of
each solution of
the dynamic optimal boundary control problem with
integer constraints is equal to
the integer part of a solution of the static problem.
A numerical verification  is given
for a control problem
in gas pipeline operations.




AMS:
35L04, 49K20, 90C46


\section{Introduction} \label{intro}
Boundary control problems for
systems governed by hyperbolic partial differential
equations
(pdes) appear in many applications,
for example  water  or gas  transportation systems,
see  e.g.
\cite{BastinCoron2016}.
 Applications of this type
give rise to
optimal boundary control problems,
where an objective function models
the aims of the control design.
%
%
%
In these control problems, it makes sense to
consider
finite time horizons.
An overview of the quadratic optimal control of hyperbolic partial differential equations
is given in \cite{lasiecka}.
In this paper we are interested in
results about the structure of the optimal boundary controls
in the spirit of the turnpike theory.
Since
the evolution of the state in time is governed
by the hyperbolic pde,
we call the corresponding optimal controls
the dynamic optimal controls.
The
turnpike phenomenon can be summarized in the statement
that
{\em in large time} intervals,  {\em the optimal state,  control and adjoint vector
remain most of the time close to an optimal steady-state}
(see \cite{trezuzhang}).
This means that
in order to get an idea of the
dynamic optimal controls,
it make sense to look
at the solution of
a certain static boundary control
problem first,
where all time derivatives are set to zero.
This static control problem determines optimal static states.
Let us call the corresponding optimal control the
static optimal control.
Our aim is
to give a bound for the difference between
the  static optimal control (that is independent of time)
and the dynamic optimal control.

Turnpike theory has originally been discussed in economics,
see \cite{samuelson}. Turnpike properties for discrete--time optimal
control problems
have been studied in \cite{dammgruene}, \cite{gruene}.
Recently there has been some interest in
the study of the turnpike phenomenon for
infinite dimensional optimal control problems,
in particular with systems governed by pdes,
see for example \cite{porretta} for the linear case,
\cite{porretta2} for the parabolic semilinear case
and \cite{trezushape} for optimal shape design with the heat equation.
Problems with infinite-dimensional control systems
have also been studied in \cite{zaslavski}.
The results can be applied to control systems with
distributed control.
In this paper, we consider boundary control systems
that are governed by hyperbolic   pdes.
Optimal Neumann boundary control problems
for systems   governed by the wave
equation have been studied in \cite{gutrezu}.
A review of turnpike results for wave equations is given in \cite{Zuazua2017}.

{ 
In \cite{dammgruene}, \cite{gruene} and the  recent contributions
\cite{faulwasser}, \cite{trelat}
on turnpike theory,
dissipativity plays an essential role.
In \cite{faulwasser}, system states in a
finite-dimensional space are considered.
In \cite{trelat} infinite-dimensional states
spaces are considered and the
control acts as a distributed control in
the partial differential equation
in such a way that mild solutions are
well-defined.
In \cite{trelat}, both
integral- and measure--turnpike properties
are considered.
In this paper, we consider
integral--turnpike properties.
Here we mention that the optimal control problems that we consider in this paper
(i.e. (\ref{ocplambda}), (\ref{ocplambdaswitch}))
satisfy a  dissipation inequality
as defined in  \cite{trelat}
if there exists a number
$\Xi_0>0$ such that
the supply rate function $\omega(u)$
as a function of the control $u$
satisfies the inequality
$\int_0^t \omega(u(\tau))\, d\tau
\geq \Xi_0 \|u\|_{(L^2(0,\, t))^2}^2$.
Then we can find a number $\Xi_1>0$ such that
for a state $r\in (L^2(0, L))^2$
the storage function $S( r ) =  \Xi_1\,  \| r \|_{(L^2(0, L))^2}^2$
satisfies a dissipation inequality.

%
}

In this paper we study optimal Dirichlet boundary control problems for systems
that are governed by linear  $2\times 2$ hyperbolic pdes.
%
A similar problem of optimal boundary control is studied in \cite{hasan},
but the turnpike phenomenon is not considered. Our motivation for this setting
is to obtain structural insights for the optimal control of gas flow in pipelines.
Also the linearized Saint-Venant equations,  that can be used as a  model for the
flow of water through channels, have this form, see \cite{onlyapunov}.
In these applications, also binary decisions are important to model
for example the decision to open or close a certain valve or to switch on or off a
control device such as a compressor \cite{isiam}.
This motivates the study of optimal boundary control problems with integer
constraints.

This paper has the following structure. In Section~\ref{mainresults} we
define the system dynamics and present our results for optimal boundary
control problems for unconstrained and integer-constrained cases.
In Section~\ref{derivation} a proof is given for the main result concerning
the unconstrained case. In Section~\ref{derivationint} we prove our result
for the integer-constrained case. In Section~\ref{applicationsection} we
discuss an application of our results to optimized operation of gas pipelines
and provide a numerical verification. In Section~\ref{conclusions}
we present conclusions.

\section{Problem definition and main results} \label{mainresults}

\subsection{Hyperbolic system dynamics}
\label{hyperbolic}
The aim of this contribution is
to study the turnpike phenomenon for
systems that are governed by hyperbolic
pdes.
We consider a $2\times 2 $ system in diagonal form.
Let a length $L>0$ and a time interval $[0,T]$ be given.
Let  $d_-$ and $d_+$ be real-valued
continuously differentiable functions
that are defined on the space-interval $[0,\,L]$ such that
for all $x\in [0,\,L]$ the inequality
$d_-(x) < 0 < d_+(x)$ holds.
Define the ($x$-dependent) diagonal matrices
\[
D(x) =
\left(
\begin{array}{cc}
d_+(x)&  0
\\
0   &  d_-(x)
\end{array}
\right),\;\;
D'(x) =
\left(
\begin{array}{cc}
d_+'(x) &  0
\\
0   &  d_-'(x)
\end{array}
\right)
.
\]
For all $x\in [0,\, L]$, let $M(x)$ denote a
$2\times 2$ matrix
that depends continuously on $x$.
Let $\eta_0\leq 0$ be  a real number.
For real numbers $\mu_+$, $\mu_-$ define
the matrix
\begin{equation}
\label{edefinition2018}
E(x) =
\left(
\begin{array}{cc}
\exp(- \mu_+ \,x)
&  0
\\
0  &  \exp( \mu_- \, \,x)
\end{array}
\right).
\end{equation}
Assume that
there exist $\mu_+>0$, $\mu_->0$
and $\nu_a < 0$
such that  for all $x\in [0,\, L]$
\begin{equation}
\label{corondelta}
\sup_{v:\, v^\top E(x)\, v=1}
v^\top \left[
\frac{d}{dx}\left(E(x) D(x) \right)
- |\eta_0| \left( E(x)  M(x) + M(x)^\top E(x) \right) \right] v  \leq \nu_a.
\end{equation}
Moreover, assume that
there exist
$\mu_+<0$, $\mu_-<0$ and $\nu_0 > 0$
such that  for all $x\in [0,\, L]$
\begin{equation}
\label{corondeltb}
\inf_{v:\, v^\top E(x)\, v=1}
v^\top \left[
 E'(x) D(x) - E(x)  D'(x)
+ |\eta_0| \left( E(x)  M(x) + M(x)^\top E(x) \right) \right] v \geq \nu_0.
\end{equation}

\begin{remark}
If $M(x)$ is a diagonal matrix or if $|\eta_0|$ is sufficiently small
or if $L>0$ is sufficiently small,
(\ref{corondelta}) and (\ref{corondeltb}) hold.
If $M^\top=M$, (\ref{corondelta}) and (\ref{corondeltb}) are
equivalent with
$\nu_0 = -\nu_a$.
\end{remark}

%

%
Consider the linear hyperbolic pde
\begin{equation}
\label{pde}
r_t + D\, r_x =
\eta_0\, M\, r
\end{equation}
where for $x\in (0,\, L)$ and $t\in (0,T)$, the state is given by
$r(t,\,x)=
\left(
\begin{array}{r}
r_+(t,\,x)
\\
r_-(t,\,x)
\end{array}
\right).
$
To obtain an initial boundary value
problem, in addition to (\ref{pde})
we consider the initial condition
$r(0,\,x)=0$ for $x\in (0,\, L)$ at the time $t=0$
and
for $t\in (0,T)$
the Dirichlet
 boundary conditions
$
r_+(t,\,0)= u_+(t),\;
r_-(t,\, L)= u_-(t)
$
with boundary controls
$u_+$, and $u_-$  in  $L^2(0,T)$.
%
The resulting initial boundary value problem
\begin{equation}
\label{linearizedsystem}
\left\{
\begin{array}{l}
r(0,\,x)=0,
\\
r_t + D\, r_x = \eta_0\,M\, r,
\\
r_+(t,\,0)= u_+(t),
\\
r_-(t,\, L)= u_-(t),
\end{array}
\right.
\end{equation}
has a solution
$r\in C([0,T], L^2((0,\, L); {\mathbb R}^2))
$.
Moreover, for the boundary traces
of the solution we have
$r_+(\cdot, L)$, $r_-(\cdot,\,0)\in L^2(0,\, T)$.
This follows  with a Picard iteration
along the characteristic curves
similar as in \cite{higdon}, \cite{hoermander}.

\subsection{Unconstrained optimal boundary control}
\label{dynamic optimal boundary control problem}

In this section we define a dynamic optimal boundary control problem for
(\ref{pde})
along with a corresponding static optimal control problem and state the
result relating the solutions of the two problems.

%
%

For $x= (x_+,\, x_-)^\top \in {\mathbb R}^2$,
we use the notation $\| x \|_{{\mathbb R}^2} = \left| x_+^2 + x_-^2 \right|^{1/2}$.
Let strictly convex quadratic functions $f_0$ and $f_L$
be given,
that is for $z \in  {\mathbb R}^2$ we have  $f_0(z)= \tfrac{1}{2} z^\top A_0 z + c_0^\top z$,
$f_L(z)= \tfrac{1}{2} z^\top A_L z + c_L^\top z$,
with  symmetric positive definite $2\times 2$ matrices $A_0$, $A_L$
and vectors $c_0$, $c_L\in {\mathbb R}^2$.
%
Define the Hilbert space
$
H= L^2(0,\, T) \times L^2(0,\, T).
$
For $u=(u_+,u_-) \in H$ and $R=(R_+,R_-)\in H$,
define
\begin{equation}
\label{objectiveneu}
J(u,\, R)
=\int_0^T
f_0( u_+(t),\, R_-(t))
+
f_L( u_-(t),\, R_+(t))
\, dt
.
\end{equation}

{  
\begin{remark}
The assumption of strict convexity of
$f_0$ and $f_L$ can be slightly relaxed.
Only the strict convexity with respect to the control is essential.
Our results also hold if for $q$ as defined in (\ref{qdefinition}) there exists
a constant $\kappa>0$ such that (\ref{qungleichung}) holds.
\end{remark}
}
The choice of the objective function $J(u,\, R)$ is motivated by transportation systems such as gas pipelines,
see Section~\ref{applicationsection}.
%
%
We consider the  dynamic optimal control problem
\begin{equation}
\label{ocplambda}
\left\{
\begin{array}{l}
\min_{u
\in ( L^2(0,\, T))^2} J(u,\, (r_+(\cdot, L),\,r_-(\cdot,\,0))  )\;
\\
\mbox{\rm subject to
(\ref{linearizedsystem})}.
\end{array}
\right.
\end{equation}
With a slight abuse of notation, in the sequel we write
$J(u,\, r)$ instead of
\[J(u,\, (r_+(\cdot, L),\,r_-(\cdot,\,0))  ).\]
Our assumptions imply that the objective function
grows as fast as some real constant multiplied with
$\|u_+\|_{L^2(0,\, T)}^2 + \|u_-\|_{L^2(0,\, T)}^2 $.
Hence the existence of an optimal control follows with
the Direct Method of the Calculus of Variations
by considering a minimizing sequence and
going to a weakly converging subsequence.

In the corresponding static optimal control problem the
initial boundary value problem (\ref{linearizedsystem}) is
replaced by  the boundary value problem
\begin{equation}
\label{staticlinearizedsystem}
\left\{
\begin{array}{l}
D \,R^{(\upsigma)}_x(x) =  \eta_0 \, M\, R^{(\upsigma)}(x),
\\
R_+^{(\upsigma)}(0)= u_+^{(\upsigma)},
\\
R_-^{(\upsigma)}(L)= u_-^{(\upsigma)},
\end{array}
\right.
\end{equation}
with $x\in (0,\, L)$ and
$u^{(\upsigma)} = \left(
 u_+^{(\upsigma)},\,
u_-^{(\upsigma)}
\right)^\top
\in \mathbb R^2$.
Define the objective function
\begin{equation}
\label{objectiveneustatic}
J_0(u^{(\upsigma)},\, R^{(\upsigma)}(x)) =
f_0(u^{(\upsigma)}_+ ,\, R_-^{(\upsigma)}( 0))
+
f_L( u^{(\upsigma)}_-,\, R_+^{(\upsigma)}(L))
%
.
\end{equation}
The static optimization problem that
corresponds to the dynamic
problem (\ref{ocplambda})
is
\begin{equation}
\label{ocplambdastatic}
\left\{
\begin{array}{l}
\min_{u^{(\upsigma)}\in \mathbb R^2 }
J_0(u^{(\upsigma)},\, R^{(\upsigma)}(x))
\;
\\
\mbox{\rm subject to
(\ref{staticlinearizedsystem})}.
\end{array}
\right.
\end{equation}

We show in Section~\ref{derivation} that solutions of the dynamic and static problem are
related in the sense of the following turnpike result.

\begin{theorem}
\label{satz1}
Let $u^{(\upsigma)}$ denote the  optimal static control
that solves (\ref{ocplambdastatic})
and let $u^{(\updelta,\, T)}$ denote the optimal dynamic control
that solves (\ref{ocplambda}) with the finite time horizon $T>0$.
Let $r^{(\upsigma)}$ and $r^{(\updelta,\, T)}$  denote the corresponding states.
There exists a constant $\bar C>0$  that is independent of $T$ such that
for all $T>0$
\begin{equation}
\label{uniformbounded09122016}
\frac{1}{T} \, \int_0^T \left\|u^{(\updelta,\, T)}(\tau) - u^{(\upsigma)} \right\|^2_{{\mathbb R}^2} \, d\tau
\leq \frac{\bar C}{T}.
\end{equation}
Thus
for all $T>0$  we have the inequality
$
\int_0^1 \left\|u^{(\updelta,\, T)}(T\, s) - u^{(\upsigma)} \right\|^2_{{\mathbb R}^2} \, ds
 \leq \frac{\bar C}{T}.
$
Moreover, there exists a constant
$\tilde D>0$ such that
for all $T>0$
we have
\begin{equation}
\label{20042018l}
\int_{0}^{T}
\int_0^L
\left\|
r^{(\updelta,\, T)}(\tau,\, x) - r^{(\upsigma)}(x)
\right\|^2_{\mathbb R^2} \, d\tau
\leq \tilde D
.
\end{equation}

%

\end{theorem}

Theorem \ref{satz1} states that for increasing time horizon
$T\rightarrow \infty$, the average quadratic mean distance
between the optimal dynamic and the optimal static control
converges to zero with the rate $O(\frac{1}{T})$.
%
\begin{example}
Let real numbers
$R_+^{\ourdesi}$, $R_-^{\ourdesi}$
and
$\lambda \in (0,\, 1)$
be given.
Consider
\begin{equation}
\label{objectiveneu2018}
J(u,\, R)
=\int_0^T
(1 -\lambda)\,\left\|
\left(
\begin{array}{l}
R_+(t) - R_+^{\ourdesi}
\\
R_-(t) - R_-^{\ourdesi}
\end{array}
\right)
\right\|^2_{{\mathbb R}^2}
+
\lambda\,
\left\|
\left(
\begin{array}{l}
 u_+(t)
 \\
 u_-(t)
 \end{array}
\right)
\right\|^2_{{\mathbb R}^2}
 \, dt.
\end{equation}
The  objective function $J(u, \, r)$ is of the form (\ref{objectiveneu})
up to additive constants.
Since the system is hyperbolic, there exists times $t_+$, $t_-\in (0,\, T)$, such that
for $t>t_\pm$, the control value  $u_\pm(t)$ does not influence the
state $r_+(\cdot,L)$, $r_-(\cdot,0)$ respectively.
Thus definition (\ref{objectiveneu2018})
implies
that for $t>t_\pm(t)$, we have $u^{(\updelta,\, T)}_\pm(t)=0$ that is
in the last part of the time interval the
control is switched off since we did not impose
any condition on the terminal state $r_\pm(T,\,\cdot)$
(in contrast to \cite{gutrezu}).
If $D$ and $M$ are constant diagonal matrices,
we have
$u_\pm^{(\updelta,\, T)}(t)=0$ for $t> t_\pm =  T - L/|d_\pm|$
and
$u^{(\upsigma)}_\pm = \left[ \frac{1}{\tfrac{1}{\lambda} -1}   +
\exp\left(\eta_0 \tfrac{m_{\pm\pm}}{|d_\pm|}\,L \right) \right]^{-1}  R_\pm^{\ourdesi}
$.
Moreover, for
$t<t_\pm$ we have
$u_\pm^{(\updelta,\, T)}(t) = u^{(\upsigma)}_\pm$.
Hence
$\int_0^T \left\|u^{(\updelta,\, T)}(\tau) - u^{(\upsigma)} \right\|^2_{{\mathbb R}^2} \, d\tau
=
\tfrac{L}{d_+} \, |u^{(\upsigma)}_+|^2 + \tfrac{L}{|d_-|} \,  |u^{(\upsigma)}_-|^2
.
$


%

\end{example}
\subsection{Optimal boundary control problems with an integer control constraint}
\label{integergonstrainedproblems}
In the application often controls with a finite range of control values appear.
In particular, binary decisions can be modeled in this form.
For overviews on optimal control problems of this type
see \cite{Antsaklis2014}, \cite{Hante2017} and the
references therein.
In this section we show that also for these problems, the turnpike phenomenon can occur.
Let {\mblue
${\cal F} $ denote a finite set of integers that
contains zero.}
We consider the  integer constraint
\begin{equation}
\label{21122016}
u_+(t) \in
{\mblue \cal F}
\;\;\;\mbox{\rm for $t$ almost everywhere in $(0,T)$.}
\end{equation}
The controls that satisfy (\ref{21122016})
are simple functions with values in
${\mblue \cal F}$
almost everywhere.
In order to avoid chattering controls that switch infinitely often
between the values
in ${\mblue \cal F}$
(this is also called the Zeno phenomenon),
in the objective function switching costs are added that penalize
the switching.
For this purpose we use a penalty term with the total variation
\[{\rm Var}(u_+)= \int_0^T d\left| u_+ \, \right|
=
\sup\limits_P \sum_i \left|u_+(t_{i+1}) - u_+(t_i)\right|,
\]
where the supremum is over all possible finite partitions $P$ of $[0,\,T]$.
%

In order to make the discussion more concise, we assume for the integer constrained case
that $J$ is as in (\ref{objectiveneu2018}). Let a penalty parameter $\nu>0$ be given.
Consider the dynamic optimal boundary control problem with integer control constraint
\begin{equation}
\label{ocplambdaswitch}
\left\{
\begin{array}{l}
\min_{u\in ( L^2(0,\, T))^2} J(u,\, r)+
\nu \, T\,{\rm Var}(u_+)
\;
\\
\mbox{\rm   subject to
$(u,\,r)$ solves
(\ref{linearizedsystem})}
\,\;{\mbox{\rm and $u_+$ satisfies}}\;\; (\ref{21122016}).
\end{array}
\right.
\end{equation}
The additional switching--cost term in
the objective functions
penalizes the number of switchings
between the values
in ${\mblue \cal F}$.
Existence of optimal solutions then follows from a compactness argument similar
as in \cite{HLS}.
Let
%
$
\omega(T)
$
denote the optimal value of the dynamic optimal control problem (\ref{ocplambdaswitch}).
The corresponding static optimal control problem  with integer constraint is
\begin{equation}
\label{ocplambdastaticswitch}
\left\{
\begin{array}{l}
\min_{
u^{(\upsigma)}_+\in  {\mblue \cal F},\,
u^{(\upsigma)}_-\in {\mathbb R},
\,
 \,R^{(\upsigma)}\in ( L^{2}(0,\, L))^2} J_0(u^{(\upsigma)},\, R^{(\upsigma)})\;
\\
\mbox{\rm subject to
(\ref{staticlinearizedsystem}) }
\end{array}
\right.
\end{equation}
with
$J_0(u^{(\upsigma)},\, R^{(\upsigma)}) = (1 -\lambda)\,\left\|R^{(\upsigma)} - R^{\ourdesi}\right\|^2_{{\mathbb R}^2}
+
\lambda\,
\left\|
u^{(\upsigma)}
\right\|^2_{{\mathbb R}^2}
$.
In the objective function of the static problem (\ref{ocplambdastaticswitch}),
the switching cost does not appear.
 If we insert the zero control $(u_+(t),\, u_-(t))=(0,0)$ in
 the objective function, the switching constraint (\ref{21122016})
 is satisfied and  we also obtain an upper bound
 for the optimal value $\omega(T)$. Since the zero control
 generates  the zero state, we have
$
\omega(T) \leq
(1 -\lambda)
\, T \, \|R^{\ourdesi}\|^2_{\mathbb R^2}
$
.
For the optimal dynamic control
$u^{(\ast)}$ that solves
(\ref{ocplambdaswitch})
this yields
 $
 {\rm Var}(u_+^{(\ast)})
 \leq \frac{1}{\nu} \, (1 -\lambda)
\, \|R^{\ourdesi}\|^2_{\mathbb R^2}$.
Hence if
\begin{equation}
\label{16012017}
\nu
>(1 -\lambda)
\, \|R^{\ourdesi}\|^2_{\mathbb R^2}
\end{equation}
the optimal control $u_+^{(\ast)}$ at $x=0$ is constant.
In this case   $\omega(T)$ is equal to the optimal value of the problem
\begin{equation}
\label{ocplambdaswitchnularge}
\left\{
\begin{array}{l}
\min_{u_+\in  {\mblue \cal F},\, u_-\in  L^2(0,\, T)} J(u,\, r)
\\
\mbox{\rm   subject to
$(u,\, r)$ solves
(\ref{linearizedsystem})}.
\end{array}
\right.
\end{equation}
If a given   value of $u_+\in {\mblue \cal F}$
is fixed in
(\ref{ocplambdaswitchnularge}),
we obtain
an optimal boundary
control problem with a time--dependent control
$u_-(t)$ at $x=L$ and constant boundary control at $x=0$.
The turnpike results
from Section \ref{dynamic optimal boundary control problem}
can be adapted to this situation.



In Theorem \ref{satzintegerturnpike} we state
 that for sufficiently large values of $\nu$, that is
if (\ref{16012017}) holds,
the solution of  (\ref{ocplambdaswitch})
%
and the solution of the corresponding static problem
(\ref{ocplambdastaticswitch})
are related by the turnpike phenomenon.
\begin{theorem}
\label{satzintegerturnpike}
Assume that $\nu$ is sufficiently large
in the sense that (\ref{16012017}) holds.
%
Let $u^{(\updelta,\,T)}\in (L^2(0,\,T))^2$ denote
a solution of the optimal dynamic control problem
(\ref{ocplambdaswitch}).


There exists
a constant $\bar C>0$ that is independent of $T$
 such that for
$T>0$ sufficiently large,
there exists a solution
$u^{(\upsigma)}$ of the optimal static control problem
(\ref{ocplambdastaticswitch})
with
$u^{(\updelta,\,T)}_+(t) = u^{(\upsigma)}_+$ for all $t\in [0,\, T]$
%
and
%
\begin{equation}
\label{uniformbounded09122016minushaupt}
\frac{1}{T}\, \int_0^{T} \left\|u^{(\updelta,\,T)}(t) - u^{(\upsigma)}\right\|^2_{{\mathbb R}^2}  \, d\,t
\leq \frac{\bar C}{T}
.
\end{equation}
Moreover, for the corresponding optimal states
(\ref{20042018l}) holds.
\end{theorem}
%
Due to the integer constraint
 in general the solutions of (\ref{ocplambdaswitch})
and
(\ref{ocplambdastaticswitch})
are not uniquely determined.
Theorem \ref{satzintegerturnpike} implies that
if the solution of the static optimal control problem (\ref{ocplambdastaticswitch})
is unique,
for all sufficiently large time horizons $T>0$,
the first component of the dynamic optimal control is
independent of $t$ and $T$.
%
The proof of 
Theorem  \ref{satzintegerturnpike}
is presented in Section \ref{derivationint}.

\section{Analysis for the unconstrained case} \label{derivation}

\subsection{An adjoint operator}
\label{adjointoperator}

For a given  time $T>0$, we define the operator
$
F_T\, (u)
$
that maps the boundary control $u = (u_+(\cdot),\, u_-(\cdot))\in H$ to
the boundary trace
$(r_+(\cdot, L),\, r_-(\cdot,\, 0) )  $
of the solution of the linear initial boundary value problem
(\ref{linearizedsystem}).
Thus we have
$
F_T\,
u
=
\left(
\begin{array}{r}
r_+(\cdot, L)
\\
r_-(\cdot,\, 0)
 \end{array}
\right).
$

For a given  time $T>0$
and a given initial state $h_0 \in (L^2(0,\, L))^2$
 we define the operator
$
G_T\, (u,\, h_0)
$
that maps the boundary control $u = (u_+,\, u_-)\in
H$ and $h_0$ to the solution
$(r_+,\, r_- ) \in (L^2((0,\, T)\times (0,\, L)))^2 $
of the initial boundary value problem
\begin{equation}
\label{system201819041533}
\left\{
\begin{array}{l}
r(0,\,x)=h_0(x),
\\
r_t + D\, r_x = \eta_0\,M\, r,
\\
r_+(t,\,0)= u_+(t),
\\
r_-(t,\, L)= u_-(t).
\end{array}
\right.
\end{equation}

\begin{lemma}
\label{exponentialdecayconservationlemma19042018}
Let $u \in H$ be given.
There exists a constant
$\tilde C>0$
that is independent of $T$ such that
for all $T>0$ we have
\begin{equation}
\label{19042018a}
\int_{0}^{T}
\int_0^L
\left\|
(G_T(u,\, h_0))
(\tau,\, x)
\right\|^2_{\mathbb R^2}\, dx  \, d\tau
\leq \tilde C
\left(
\left\| u\right\|^2_{H}
+
\left\| h_0 \right\|^2_{(L^2(0,\, L))^2}
\right)
.
\end{equation}
\end{lemma}

%
%
\begin{proof}
Let real numbers $\mu_+>0$
 and
$\mu_->0$ be given.
Define the matrix $E(x)$ as in
(\ref{edefinition2018}).
For $t>T$ we define
$u_\pm(t)=0 $.
For $t>0$ consider the Lyapunov functional
\begin{equation}
\label{lyapunovdefinition20042018}
E_a(t)
 =
 \frac{1}{2}
 \int_t^{t+1}
\int_0^L
\left( r(\tau,\, x)  \right)^\top\, E(x)\,
  r(\tau,\, x) \, dx \, d\tau
  \end{equation}
where $r$ is
the solution (\ref{system201819041533})
for $t>0$.
For the time derivative of $E_a$
we obtain
\begin{eqnarray*}
E_a'(t) & = &
2\int_{t}^{t+1}
\int_0^L
-\left( r(\tau,\, x)\right)^\top
\,
E(x)\,D(x)
\,
\left( r(\tau,\, x)  \right)_x
\\
& - &
\frac{1}{2}\,
\left( r(\tau,\, x)\right)^\top
\,
E(x)\,D'(x)
\,
\left( r(\tau,\, x)  \right)
- \left( r(\tau,\, x)  \right)^\top \, M_1(x)
\,\left( r(\tau,\, x)  \right)
\,
dx\,
d \tau
\end{eqnarray*}
with the symmetric matrix $M_1$ defined as
\begin{equation}
\label{edefinition}
M_1(x) = \frac{|\eta_0|}{2} \left[ E(x)\, M(x)+ M(x)^\top \, E(x) \right]  - \frac{1}{2} \, D'(x)\, E(x) .
\end{equation}

Integration by parts yields
$E_a'(t) = T_B  + T_R$ with
\begin{eqnarray*}
T_B
& := &
- \int_{t}^{t+1}
\left( r(\tau,\, x) \right)^\top
\,
E(x)\,D(x)
\,
\left( r(\tau,\, x)   \right)|_{x=0}^L
d\tau
\end{eqnarray*}
and
\begin{eqnarray*}
T_R & := & \int_{t}^{t+1}
\int_0^L
\left( r(\tau,\, x)  \right)^\top
\,
E'(x)\,D(x)
\,
\left( r(\tau,\, x)  \right)\,dx \,
d\tau
\\
& &
- 2 \,
\int_{t}^{t+1}
\int_0^L
 \left( r(\tau,\, x) \right)^\top \, M_1(x)
\,\left( r(\tau,\, x)   \right)
\,
dx \, d\tau.
\end{eqnarray*}
Define
$
\xi(t)= \max\left\{ |d_-(L)|\,{\rm e}^{\mu_- \,L}, \; d_+(0)     \right\}
\int_{t}^{t+1}
\left\|
\left(u_+(\tau),\, u_-(\tau) \right)
\right\|_{\mathbb R^2}^2 \, d\tau
$
.
Then we have
$T_B \leq \xi(t)$.
Due to
assumption (\ref{corondelta})
we can choose  $\mu_+>0$ and $\mu_->0$  such that
\begin{eqnarray*}
T_R &  \leq  &
\nu_a \, E_a(t).
\end{eqnarray*}
 This yields
 \[E_a'(t)  = T_B  + T_R \leq  \nu_a \, E_a(t) + \xi(t).\]
By Gronwall's Lemma this implies
for all $j\in \{0,1,2,3,...\}$ the inequality
\[
0 \leq E_a(j+1) \leq  \exp(\nu_a) \, E_a(j) +
\int_j^{j+1}
\xi(t)\, dt
.
\]
By induction, this implies
for all $N\in \{0, \,1,\, 2,\, 3,...\}$
\[
E_a(j+1) \leq  \exp(\nu_a\, (j+1)) \, E_a(0) +  \sum_{k=0}^{j} \exp(\nu_a\, (j -k))
\,
\int_k^{k+1}
\xi(t)\, dt.
\]
Hence  we obtain
\begin{equation}
\label{19042018b}
\sum_{j=0}^N
E_a(j) \leq  \left(\sum_{j=0}^\infty \exp(\nu_a\, j)\right)  \,\left(  E_a(0) +  \sum_{j=0}^N
\int_j^{j+1} \xi(t)\, dt \right).
\end{equation}
We have
\[
\sum_{j=0}^N
\int_j^{j+1} \xi(t)\, dt
=
\int_0^{N + 1}  \xi(t)\, dt
\leq
\int_0^\infty  \xi(t)\, dt.
\]
Define
$\tilde K=\max\left\{ |d_-(L)|\,{\rm e}^{\mu_- \,L}, \; d_+(0)     \right\}$.
The definition of $\xi$ implies that
\[
\tfrac{
\int_0^\infty  \xi(t)\, dt
}{\tilde K}
=
\int\limits_0^\infty
\int\limits_{t}^{t+1}
\left\|
u(\tau)
\right\|_{\mathbb R^2}^2 \, d\tau
\, dt
=
\int\limits_0^T
\int\limits_0^1
\left\|
u(\tau+t)
\right\|_{\mathbb R^2}^2 \, d\tau
\, dt
\leq
\int\limits_0^T
\left\|
u(\tau)
\right\|_{\mathbb R^2}^2 \,  d\tau.
\]
Thus we have
$
\int_0^\infty  \xi(t)\, dt
\leq
\tilde K
\,
\left\|
u
\right\|_{H}
$.
Hence  (\ref{19042018b})   implies $\sum\limits_{j=0}^\infty E_a(j) <\infty$.
There exists a number
$\tilde C_0$ such that
\[
E_a(0)=\frac{1}{2}
 \int_0^{1}
\int_0^L
\left( r(\tau,\, x)  \right)^\top\, E(x)\,
  r(\tau,\, x) \, dx \, d\tau
  \leq
  \tilde C_0\left(
\left\| u\right\|^2_{H}
+
\left\| h_0 \right\|^2_{(L^2(0,\, L))^2}
\right).
\]
We have
$
\sum\limits_{j=0}^\infty E_a(j)
=
 \frac{1}{2}
 \int_0^{\infty}
\int_0^L
\left( r(\tau,\, x)  \right)^\top\, E(x)\,
  r(\tau,\, x) \, dx \, d\tau
  $.
  Hence
  (\ref{19042018b})
  yields
  \[
 \int\limits_0^{\infty}
\int\limits_0^L
\left( r(\tau,\, x)  \right)^\top\, E(x)\,
  r(\tau,\, x) \, dx \, d\tau
  \leq\frac{2}{1-{\rm e}^{\nu_a}}
  \left(
  \tilde C_0
  +
  \tilde K
  \right)
  \left(
\left\| u\right\|^2_{H}
+
\left\| h_0 \right\|^2_{(L^2(0,\, L))^2}
\right)
.\]
  This implies (\ref{19042018a}).
\end{proof}

Using Lemma \ref{exponentialdecayconservationlemma19042018}
and integration by parts
we can prove Lemma \ref{ftlemma1}.
\begin{lemma}
\label{ftlemma1}
The  operator $F_T$ is uniformly bounded as an operator
in
the Hilbert space $H$
that is
there exists a constant $C_N>0$ that is independent of $T$ such that
 for the corresponding operator norm of $F_T$ for all $T>0$ we have
\begin{equation}
\label{ftnorm}
\|F_T \|\leq C_N.
\end{equation}
\end{lemma}




For the analysis of the boundary
control problem, the study of the  adjoint operators for $F_T$
defined at the end of
\ref{adjointoperator}
is essential.
The adjoint operator  $F_T^\ast$ that satisfies the equation
\begin{equation}
\label{adjointdesiredequation}
\int_0^T
\left\langle F_T\,(u)(t),
\,
z_T(t)
\right\rangle_{\mathbb R^2}
\,dt
=
\int_0^T
\left\langle
u(t),
\,
F_T^\ast(z_T)(t)
\right\rangle_{\mathbb R^2}
\,dt
\end{equation}
for all $z_T \in {\cal D}(F_T^\ast)= H$
where
$\langle\cdot,\,\cdot \rangle_{{\mathbb R}^2}$ denotes the usual
scalar product in ${\mathbb R}^2$.
Due to (\ref{ftnorm}) we have
the inequality
\begin{equation}
\label{ftnormadjoint}
\|F_T^\ast \|\leq C_N.
\end{equation}
Similar as in  \cite{coron}, we determine $F_T^\ast$ in the following lemma.
\begin{lemma}
\label{adjoint24052018}
For  $z_T = (z_+^T, \, z_-^T)\in {\cal D}(F_T^\ast)= H$,  
define $z = ( z_+(\cdot),\, z_-(\cdot))  $
as the solution of the adjoint system
(where $(t,\,x)\in (0,\, T)\times (0,\, L)$)
\begin{equation}
\label{adjointsystem}
\left\{
\begin{array}{l}
z(T,\,x)=0,\, x\in (0,L),
\\
z_t(t,\, x) + D\, z_x(t,\, x) = - \eta_0\,M(x)^\top\, z(t,\,x) - D'(x)\, z(t,\,x) ,
\\
z_+(t,\, L)=  \frac{1}{d_+(L)}\,z_+^T(t),
\\
z_-(t,\,0)=   \frac{1}{|d_-(0)|}\,   z_-^T(t).
\end{array}
\right.
\end{equation}
Then we have
\begin{equation}
F_T^\ast\,\left(
\begin{array}{r}
z_+^T(\cdot)
\\
  z_-^T(\cdot)
 \end{array}
\right)
=
\left(
\begin{array}{r}
d_+(0) \, z_+(\cdot,\, 0)
\\
|d_-(L)|\, z_-(\cdot, L)
 \end{array}
\right).
\end{equation}
\end{lemma}
For our proof of the turnpike result,
the fact that the operator norm
of $F_T$ is uniformly bounded with
respect to $T$ is essential.
Moreover, it is important that with
boundary controls that are zero, the system state
decays exponentially with time
for the forward system in the forward direction
and for the adjoint system with time going backwards.

\subsection{Necessary optimality conditions for the dynamic problem}
\label{necessaryoptimalityconditions}
In order to determine the structure of the  dynamic optimal control
 $ u^{(\updelta,\, T)}$ that solves
 (\ref{ocplambda})
we look at the necessary optimality conditions.
For all $u,\, r\in H$
 that satisfy (\ref{linearizedsystem}), we have
\begin{equation}
\label{09122016a}
 J(u,\, r)
 =
 \int_0^T
 f_0(u_+(t),\, (F_T u)_-(t) \,)
 +
 f_L(u_-(t),\, (F_Tu)_+(t) \, )
 \,dt
.
\end{equation}
Let $u = u^{(\updelta,\, T)} + \delta^{(1)}$
with a control variation $\delta^{(1)}\in H$. 
Let $R = R^{(\updelta,\, T)} + \delta^{(2)}$
denote the  corresponding
state,
that is we have
\[
\left\{
\begin{array}{l}
\delta^{(2)}(0,\,x)=0,\, x\in (0,L),
\\
\delta^{(2)}_t + D\, \delta^{(2)}_x = \eta_0\, M\,\delta^{(2)},
\\
\delta^{(2)}_+(t,\,0)= \delta^{(1)}_+(t),
\\
\delta^{(2)}_-(t,\, L)= \delta^{(1)}_-(t),
\end{array}
\right.
\]
or
$F_T\,(\delta^{(1)})
=
\left(
\delta^{(2)}_+(\cdot,\,L),
%
\delta^{(2)}_-(\cdot,\,0)
\right)^\top.$
%
Since
$J$ is convex and we have
%
%
\begin{eqnarray*}
& &
 J(u^{(\updelta,\,T)} + \delta^{(1)},\,
R^{(\updelta,\,T)} + \delta^{(2)})
\geq J(u^{(\updelta,\,T)},\, R^{(\updelta,\,T)})
\\
& + &
\langle  A_0 \left(
 \begin{array}{c}
 u_+^{(\updelta,\,T)}
 \\
  (F_T u^{(\updelta,\,T)})_-
  \end{array}
  \right) + c_0,\, \left(
  \begin{array}{c}
  \delta^{(1)}_+
  \\
   (F_T\delta^{(1)})_-
   \end{array}
   \right) \rangle_{H}
+
\langle
 A_L
 \left(
 \begin{array}{c}
  u_-^{(\updelta,\,T)}
  \\
   (F_T u^{(\updelta,\,T)})_+
   \end{array}
   \right) + c_L,\,
   \left(
 \begin{array}{c}
 \delta^{(1)}_-
 \\
  (F_T\delta^{(1)})_+
   \end{array}
   \right)
   \rangle_{H}
.
\end{eqnarray*}
Define the vectors $v_1=(c_{0,1},\, c_{L,1})^\top$,
$v_2 = (c_{L,\,2},\, c_{0,\,2})^\top$ where $c_0=(c_{0,1}, \, c_{0, 2})^\top$
and $c_L=(c_{L,1}, \, c_{L, 2})^\top$.
Then we have
\[
\langle
c_0,\, (\delta^{(1)}_+,\, (F_T\delta^{(1)})_- )^\top \rangle_{H}+  \langle c_L,\, (\delta^{(1)}_-,\, (F_T\delta^{(1)})_+ )^\top\rangle_{H}
=
\langle
v_1 + F_T^\ast v_2, \, \delta^{(1)}\rangle_{H}
.
\]
Define 
$
{\cal M}_1 = \left( \begin{array}{ll} a^0_{11} & 0 \\ 0 & a^L_{11} \end{array} \right)
$,
$
{\cal M}_2 = \left( \begin{array}{ll} 0 & a^0_{12}  \\  a^L_{12} & 0 \end{array} \right)$,
${\cal M}_3 = \left( \begin{array}{ll} 0 & a^L_{12}  \\  a^0_{12} & 0 \end{array} \right)$,
${\cal M}_4 = \left( \begin{array}{ll} a^L_{22} & 0 \\ 0 & a^0_{22} \end{array} \right)$
where $A_0 = \left( \begin{array}{ll} a^0_{11} & a^0_{12} \\ a^0_{12} & a^0_{22} \end{array} \right)$,
$A_L = \left( \begin{array}{ll} a^L_{11} & a^L_{12} \\ a^L_{12} & a^L_{22} \end{array} \right)$.
Then we have
\[
 J(u^{(\updelta,\,T)} + \delta^{(1)},\,
R^{(\updelta,\,T)} + \delta^{(2)})
\geq
J(u^{(\updelta,\,T)},\, R^{(\updelta,\,T)})
 \]
 \[
+
\langle
{\cal M}_1 u^{(\updelta,\,T)} + {\cal M}_2 F_Tu^{(\updelta,\,T)} + v_1  + F_T^\ast\left({\cal M}_3 u^{(\updelta,\,T)} + {\cal M}_4 F_Tu^{(\updelta,\,T)} + v_2\right),\delta^{(1)})
\rangle_{H}
.
\]

This implies the  optimality conditions that are stated in the following lemma.
Due to the convexity of the problem, they are necessary and sufficient
(see also the Lagrange multiplier rule, as for example in  \cite{kunisch}).
\begin{lemma}
\label{neccessaryoptimalityconditionsdynamic}
The control  $u^{(\updelta,\,T)}$ is a solution of the dynamic optimal control problem
(\ref{ocplambda})
if and only if the optimality system
\begin{equation}
\label{optimalitysystem}
{\cal M}_1 u^{(\updelta,\,T)} + {\cal M}_2 F_Tu^{(\updelta,\,T)} + v_1  + F_T^\ast\left({\cal M}_3 u^{(\updelta,\,T)} + {\cal M}_4 F_Tu^{(\updelta,\,T)} + v_2\right)
%
=
0
\end{equation}
holds.
By the definition of $F_T$ and
the representation of $F_T^\ast$ from Lemma \ref{adjoint24052018},
this means that there exists a multiplier $p^{(\updelta,\,T)}$
such that
for $(t,\, x) \in (0,\, T) \times (0,\, L)$ almost everywhere
we have
\begin{equation}
\label{39a}
\left\{
\begin{array}{l}
R^{(\updelta,\,T)}(0,\,x)=0,\, 
\\
R_t^{(\updelta,\,T)} + D\, R_x^{(\updelta,\,T)} = \eta_0\, M\, R^{(\updelta,\,T)},
\\
R_+^{(\updelta,\,T)}(t,\,0)= u_+^{(\updelta,\,T)}(t),
\\
R_-^{(\updelta,\,T)}(t,\, L)= u_-^{(\updelta,\,T)}(t),
\\
p^{(\updelta,\,T)}(T,\,x)=0,\, 
\\
p_t^{(\updelta,\,T)} + D\, p_x^{(\updelta,\,T)} = - \eta_0\,M^\top\,  p^{(\updelta,\,T)}  - D'\, p^{(\updelta,\,T)} ,
\\
p_+^{(\updelta,\,T)}(t,\, L) =   \frac{1}{d_+(L) }\, \left(
{\cal M}_3 u^{(\updelta,\,T)} + {\cal M}_4
 \left(
 \begin{array}{c}
R^{(\updelta,\,T)}_+(t,\,L)
\\
 R^{(\updelta,\,T)}_-(t,\,0)
 \end{array}
 \right)  + v_2
\right)_+,
\\
p_-^{(\updelta,\,T)}(t,\,0)= \frac{1}{|d_-(0)|} \,  \left(
{\cal M}_3 u^{(\updelta,\,T)} + {\cal M}_4
 \left(
 \begin{array}{c}
R^{(\updelta,\,T)}_+(t,\,L)
\\
 R^{(\updelta,\,T)}_-(t,\,0)
 \end{array}
 \right)
   + v_2
\right)_-
\end{array}
\right.
\end{equation}
and
\begin{equation}
\label{couplinga}
\left\{
\begin{array}{l}
\left(
{\cal M}_1 u^{(\updelta,\,T)} + {\cal M}_2 F_Tu^{(\updelta,\,T)} + v_1
\right)_+
+
d_+(0) \,  p_+^{(\updelta,\,T)}(t,\, 0) = 0,
\\
\left(
{\cal M}_1 u^{(\updelta,\,T)} + {\cal M}_2 F_Tu^{(\updelta,\,T)} + v_1
\right)_-
+
|d_-(L)|\,   p_-^{(\updelta,\,T)}(t,\, L) = 0.
\end{array}
\right.
\end{equation}
\end{lemma}

\subsection{An adjoint operator for the static problem}
\label{adjointstatic}
We define the static operator
$
F_{(\upsigma)}\, (u^{(\upsigma)} )
$
that maps the boundary control $u^{(\upsigma)} = (u_+^{(\upsigma)},\, u_-^{(\upsigma)} )\in {\mathbb R}^2$ to
the point
$(r_+^{(\upsigma)}(L),\, r_-^{(\upsigma)}(0) ) $,
where $r^{(\upsigma)}$ solves
 the linear boundary value problem
(for  $x\in (0,L)$)
\begin{equation}
\label{linearizedsystemstatic}
\left\{
\begin{array}{l}
r_x^{(\upsigma)} = \eta_0\, D^{-1}\,M\, r^{(\upsigma)},
\\
r_+^{(\upsigma)}(0)= u_+^{(\upsigma)},
\\
r_-^{(\upsigma)}(L)= u_-^{(\upsigma)}.
\end{array}
\right.
\end{equation}
Thus we have
\begin{equation}
F_{(\upsigma)}\,\left(
\begin{array}{r}
u_+^{(\upsigma)}
\\
 u_-^{(\upsigma)}
 \end{array}
\right)
=
\left(
\begin{array}{r}
r_+^{(\upsigma)}(L)
\\
r_-^{(\upsigma)}(0)
 \end{array}
\right).
\end{equation}

In  Lemma \ref{adjointstaticlemma} an explicit representation of
the adjoint operator  $F_{(\upsigma)}^\ast$  is given
that satisfies  for all $z \in {\mathbb R}^2$ the equation
$
\langle F_{(\upsigma)}\,(u^{(\upsigma)}),
\,
z \rangle_{\mathbb R^2}
=
\langle
u^{(\upsigma)},
\,
F_{(\upsigma)}^\ast(z) \rangle_{\mathbb R^2}$.

\begin{lemma}
\label{adjointstaticlemma}
For  $z =(z_+,\, z_-)^T\in {\mathbb R}^2$,
define $( z_+^{(\upsigma)}(\cdot),\, z_-^{(\upsigma)}(\cdot))  \in (L^{2}(0,\,L))^{2} $
as the solution of the adjoint system
\begin{equation}
\label{adjointsystemstatic}
\left\{
\begin{array}{l}
z_x^{(\upsigma)} = - \eta_0\, D^{-1}\,M^\top\, z^{(\upsigma)} -  D^{-1}\,D'\, z^{(\upsigma)},
\\
z_+^{(\upsigma)}(L)=  \frac{1}{d_+(L)}\,z_+,
\\
z_-^{(\upsigma)}(0)=   \frac{1}{|d_-(0)|}\,   z_-.
\end{array}
\right.
\end{equation}
Then we have
\begin{equation}
\label{08122016}
F_{(\upsigma)}^\ast\,\left(
\begin{array}{r}
z_+
\\
  z_-
 \end{array}
\right)
=
\left(
\begin{array}{r}
d_+(0) \, z_+^{(\upsigma)}(0)
\\
|d_-(L)|\, z_-^{(\upsigma)}(L)
 \end{array}
\right).
\end{equation}
\end{lemma}
\subsection{Necessary optimality conditions for the static
problem}
\label{necessaryoptimalityconditionsstatic}
Let $ u^{(\upsigma)}$ denote the optimal control
that solves (\ref{ocplambdastatic})
and $R^{(\upsigma)}$ the  state
generated by $ u^{(\upsigma)}$
as a solution of
(\ref{staticlinearizedsystem}).
For all $u\in {\mathbb R}^2$ and $R\in ( L^{2}(0,\, L))^2$
that satisfy
(\ref{staticlinearizedsystem})
 we have
 \[
J_0(u,\, R)
=
 f_0(u_+,\, F_{(\upsigma)} (u)_- )
 +
 f_L(u_-,\, F_{(\upsigma)} (u)_+  )
%
.
\]

Let $u = u^{(\upsigma)} + \delta^{(1)}$
with a control variation $\delta^{(1)}\in \mathbb R^2$.
Let $R = R^{(\upsigma)} + \delta^{(2)}$
with a state variation $\delta^{(2)}$
denote the  corresponding  state,
that is we have
\[
\left\{
\begin{array}{l} D\, \delta^{(2)}_x = \eta_0\, M\,\delta^{(2)},
\\
\delta^{(2)}_+(0)= \delta^{(1)}_+,
\\
\delta^{(2)}_-( L)= \delta^{(1)}_-,
\end{array}
\right.
\]
or
$\left(
\begin{array}{c}
\delta^{(2)}_+(L)
\\
\delta^{(2)}_-(0)
\end{array}
\right)
 = F_{(\upsigma)}\,(\delta^{(1)})$.
For all $p \in \mathbb R^2$ we have
\[
 J_0(
 u^{(\upsigma)} + \delta^{(1)},\,
R^{(\upsigma)} + \delta^{(2)}
 )
\geq
J_0(u^{(\upsigma)},\, R^{(\upsigma)})
 \]
 \[
+
\langle
{\cal M}_1 u^{(\upsigma)} + {\cal M}_2 F_{(\upsigma)}    u^{(\upsigma)} + v_1
+ F_{(\upsigma)} ^\ast
\left( {\cal M}_3 u^{(\upsigma)}  + {\cal M}_4 F_{(\upsigma)}  u^{(\upsigma)} + v_2\right),
\,\delta^{(1)})
\rangle_{\mathbb R^2}
.
\]
%
%
%
Hence $u^{(\upsigma)}$ can only be a {\em static} optimal control
if  the optimality system
\begin{equation}
\label{optimalitysystemstatic}
{\cal M}_1 u^{(\upsigma)} + {\cal M}_2 F_{(\upsigma)}    u^{(\upsigma)} + v_1  + F_{(\upsigma)} ^\ast\left({\cal M}_3 u^{(\upsigma)}
 + {\cal M}_4 F_{(\upsigma)}  u^{(\upsigma)} + v_2\right)
%
=
0
\end{equation}
holds.
By the definition of $F_{(\upsigma)}$ and
the representation of $F_{(\upsigma)}^\ast$ from Lemma \ref{adjointstaticlemma},
this means that there exists a multiplier $P^{(\upsigma)}$
such that we have
\begin{equation}
\label{39astatic14}
\left\{
\begin{array}{l}
D\, R_x^{(\upsigma)} = \eta_0\, M\, R^{(\upsigma)},
\\
R_+^{(\upsigma)}(0)= u_+^{(\upsigma)},
\\
R_-^{(\upsigma)}(L)= u_-^{(\upsigma)},
\\
D\, P_x^{(\upsigma)} = - \eta_0\,M^\top\,  P^{(\upsigma)}  -  D'\,P^{(\upsigma)} ,
\\
P_+^{(\upsigma)}( L)=   \frac{1}{d_+(L)}\, \left(
 {\cal M}_3 u^{(\upsigma)}  + {\cal M}_4 F_{(\upsigma)}  u^{(\upsigma)} + v_2
\right)_+,
\\
P_-^{(\upsigma)}(0)= \frac{1}{|d_-(0)|} \,  \left(
 {\cal M}_3 u^{(\upsigma)}  + {\cal M}_4 F_{(\upsigma)}  u^{(\upsigma)} + v_2
\right)_-
\end{array}
\right.
\end{equation}
and
\begin{equation}
\label{39astatic14a}
\left\{
\begin{array}{l}
\left(
{\cal M}_1 u^{(\upsigma)} + {\cal M}_2 F_{(\upsigma)}    u^{(\upsigma)} + v_1
\right)_+
+ d_+(0) \,  P_+^{(\upsigma)}(0) = 0,
\\
\left(
{\cal M}_1 u^{(\upsigma)} + {\cal M}_2 F_{(\upsigma)}    u^{(\upsigma)} + v_1
\right)_-
+
|d_-(L)|\,   P_-^{(\upsigma)}(L) = 0.
\end{array}
\right.
\end{equation}

\subsection{The static optimal control is close to optimal for the dynamic optimal control problem}
\label{turnpikemainresult}
The proof of Theorem~\ref{satz1} uses the following two auxiliary results.

\begin{lemma}
\label{lemma2017}
Let $t\geq 0$ be given. For $ y_0^{(t)}(x)\in L^2(0,\,L)\times L^2(0,L)$, consider the
initial boundary value problem (\ref{adjointsystem02062017})
for $x\in (0,L)$ and $s \in [t-1,\, t]$:
  \begin{equation}
\label{adjointsystem02062017}
\left\{
\begin{array}{l}
f(t,\,x) =  y_0^{(t)}(x),
\\
f_t(s,\,x) + D\, f_x(s,x)  = - \eta_0\,M^\top\, f(s,\,x) - D' \, f(s,\,x) ,
\\
f_+(s,\, L)=  0,
\\
f_-(s,\,0) =  0.
\end{array}
\right.
\end{equation}
There exists a constant
$\tilde C_2\geq 0$ such that
we have the inequality
\begin{equation}
\label{02062017a}
\int_{t-1}^t |d_-(L)| \, f_-(s,\,L)^2 + d_+(0) \, f_+(s,\,0)^2 \,ds
\leq
(1 + \tilde C_2)
\int_0^L \left\|y_0^{(t)}(x)\right\|_{{\mathbb R}^2}^2 \, dx.
\end{equation}
Now consider the
initial boundary value problem (\ref{adjointsystem08062017})
for $x\in (0,L)$ and $s \in [t,\, t+1]$:
  \begin{equation}
\label{adjointsystem08062017}
\left\{
\begin{array}{l}
g(t,\, x) =  y_0^{(t)}(x),
\\
g_t(s,\,x) + D\, g_x(s,\,x)  =  \eta_0\,M\, g(s,\,x) ,
\\
g_+(s,\, 0)=  0,
\\
g_-(s,\,L) =  0.
\end{array}
\right.
\end{equation}
There exists a constant
$\tilde C_3\geq 0$ such that
we have the inequality
\begin{equation}
\label{08062017aa}
\int_{t}^{t+1} |d_-(0)| \, g_-(s,\,0)^2 + d_+(L) \, g_+(s,\,L)^2 \,ds
\leq
(1 + \tilde C_3)
\int_0^L \left\|y_0^{(t)}(x)\right\|_{{\mathbb R}^2}^2 \, dx.
\end{equation}

\end{lemma}
\begin{proof}
Theorem A.4 from \cite{BastinCoron2016} implies that
there exists a constant
$\tilde C_2\geq 0$ such that
\[
 \int_{t-1}^{t}
\int_0^L
f^\top
\left(  |\eta_0|\,  \, M^\top
- \frac{1}{2} D' \right) f
 dx \, d\tau
  \geq
 - \frac{\tilde C_2}{2} \,
\left\|  y_0^{(t)} \right\|^2_{(L^2(0,\, L))^2}
.
\]
From (\ref{adjointsystem02062017})
we obtain the equation
$
f^\top\,f_t + f^\top \,D\, f_x = -\eta_0\, f^\top \, M^\top  \, f - f^\top\, D' \, f,
$
hence
\[\partial_t \left( \frac{1}{2} f^\top f\right)
= - \partial_x \left( \frac{1}{2} f^\top \, D\,  f\right)
-  \eta_0\, f^\top \, M^\top \, f
- \frac{1}{2}
f^\top \, D'\, f.
\]
 This implies
\[\partial_t \left( \frac{1}{2} f^\top f\right)
 \geq - \partial_x \left( \frac{1}{2} f^\top \, D\,  f\right)
 +
 f^\top
\left(  |\eta_0|\,  \, M^\top
- \frac{1}{2} D' \right) f.
\]
Integration with respect to the space variable $x$ yields
\begin{eqnarray*}
& &
\partial_t \int_0^L \left( \frac{1}{2} f^\top f\right)\,dx
\\
&  \geq &
 -  \left. \left( \frac{1}{2} f^\top \, D\,  f\right)\right|_{x=0}^L
 +
\int_0^L
 f^\top
\left(  |\eta_0|\,  \, M^\top
- \frac{1}{2} D' \right) f \, dx
 \\
&  = & - \frac{1}{2}  \left[
 d_-(L) \, f_-(s,\,L)^2   - d_+(0) \,   f_+(s,\,0)^2
    \right]
 +
\int_0^L
 f^\top
\left(  |\eta_0|\,  \, M^\top
- \frac{1}{2} D' \right) f \, dx
    .
\end{eqnarray*}
Due to the boundary condition in (\ref{adjointsystem02062017}),
integration with respect to
time
 yields
\begin{eqnarray*}
& &\int_{0}^L  \frac{1}{2} \|f(t,\, x)\|^2_{\mathbb R^2}\,dx
-
\int_{0}^L  \frac{1}{2} \|f(t-1,\, x)\|^2_{\mathbb R^2}\,dx
\\
& \geq &
\frac{1}{2} \, d_+(0) \, \int_{t-1}^t  f_+(s,0)^2\, d\tau
+
\frac{1}{2} \, |d_-(L)| \, \int_{t-1}^t f_-(s,L)^2\, d\tau
    +
\int_{t-1}^t
\int_0^L
 f^\top
\left(  |\eta_0|\,  \, M^\top
- \frac{1}{2} D' \right) f \, dx
\, d\tau
\\
&
\geq &
\frac{1}{2} \, d_+(0) \, \int_{t-1}^t  f_+(s,0)^2\, d\tau
+
\frac{1}{2} \, |d_-(L)| \, \int_{t-1}^t f_-(s,L)^2\, d\tau
    - \frac{\tilde C_2}{2}  \,
\left\|  y_0^{(t)} \right\|^2_{(L^2(0,\, L))^2}
.
\end{eqnarray*}
Due to the terminal condition in (\ref{adjointsystem02062017}),
this implies
(\ref{02062017a}).
%
The proof of (\ref{08062017aa}) is similar.
%
\end{proof}

With Lemma \ref{lemma2017} we can prove
Lemma \ref{lemma2016}.
\begin{lemma}
\label{lemma2016}
Let $u^{(\upsigma)} \in {\mathbb R}^2$
be a static optimal control
that solves (\ref{ocplambdastatic}).
There exists a constant $C_D>0$ that is independent of $T$
such that for all $T>0$
and with  $u^{(\upsigma)}$
considered as a constant function in the Hilbert space $H$
we have
the inequality
\begin{equation}
\label{c4assertion}
\left\|
{\cal M}_1 u^{(\upsigma)} + {\cal M}_2 F_T u^{(\upsigma)} + v_1
+ F_T^\ast\left({\cal M}_3 u^{(\upsigma)} + {\cal M}_4 F_T u^{(\upsigma)} + v_2\right)
%
%
 \right\|_{H}
\leq
C_D.
\end{equation}
\end{lemma}
\begin{proof}
Due to the optimality system (\ref{optimalitysystemstatic}) we have the equation
\begin{eqnarray*}
& &
{\cal M}_1 u^{(\upsigma)} + {\cal M}_2 F_T u^{(\upsigma)} + v_1
+ F_T^\ast\left({\cal M}_3 u^{(\upsigma)} + {\cal M}_4 F_T u^{(\upsigma)} + v_2\right)
\\
& = &
\left( {\cal M}_2
+
F_T^\ast  {\cal M}_4
\right)
 \left(F_T  - F_{(\upsigma)}\right) u^{(\upsigma)}
+
(F_T^\ast - F_{(\upsigma)}^\ast) \left( {\cal M}_3 u^{(\upsigma)}  +  {\cal M}_4 F_{(\upsigma)} u^{(\upsigma)}  + v_2 \right)
.
\end{eqnarray*}

Define
$
p_0 =
 {\cal M}_3 u^{(\upsigma)}  +  {\cal M}_4 F_{(\upsigma)} u^{(\upsigma)}  + v_2
%
$.
By definition, we have
\begin{equation}
\label{08062017a}
 \left( F_T^\ast - F_{(\upsigma)}^\ast\right) p_0
=\left(
\begin{array}{r}
d_+(0) \,\left( z_+(\cdot,\, 0) - z_+^{(\upsigma)}(0) \right)
\\
|d_-(L)|\, \left(  z_-(\cdot, L) -  z_-^{(\upsigma)}(L) \right)
 \end{array}
\right),
\end{equation}
where for $x\in (0,L)$ and $t\in (0,T)$ we have
  \begin{equation}
\label{adjointsystem1611}
\left\{
\begin{array}{l}
z(T,\,x) - z^{(\upsigma)}(x) =   - z^{(\upsigma)}(x),
\\
(z-z^{(\upsigma)})_t + D\, (z-z^{(\upsigma)}) _x = - \eta_0\,M^\top \, (z-z^{(\upsigma)})- D' \, (z-z^{(\upsigma)}),
\\
z_+(t,\, L) - z_+^{(\upsigma)}(L)=  0,
\\
z_-(t,\,0) - z_-^{(\upsigma)}(0) =  0,
\end{array}
\right.
\end{equation}
with $z^{(\upsigma)}(x) = F_{(\upsigma)}^\ast \,p_0$.
For real numbers $\mu_+$ and $\mu_-$ and $x\in [0,\, L]$,
let the matrix $E(x)$ be defined as in (\ref{edefinition2018}).
%
Similar as in \cite{BastinCoron2016},
for real numbers $\mu_+<0$ and $\mu_-<0$
consider the Lyapunov functional
\begin{equation}
\label{lyapunovdefinition}
E_0(t)
 =\frac{1}{2}
\int_0^L {\rm e}^{- \mu_+ \,x} \left( z_+(t,\, x) - z_+^{(\upsigma)}(x) \right)^2
+
{\rm e}^{ \,\mu_- \, \,x} \left( z_-(t,\, x) - z_-^{(\upsigma)}(x) \right)^2\, dx
.
\end{equation}
Using matrix and vector notation, we can write $E_0(t)$ in the form
\begin{equation*}
\label{lyapunovdefinition2017}
E_0(t)
 =\frac{1}{2}
\int_0^L \left( z(t,\, x) - z^{(\upsigma)}(x) \right)^\top
\,
E(x)
\,
\left( z(t,\, x) - z^{(\upsigma)}(x)  \right)
\, dx
.
\end{equation*}
For the time derivative of $E_0$
we obtain
\begin{eqnarray*}
E_0'(t) & = &
\int_0^L
\left( z(t,\, x) - z^{(\upsigma)}(x)  \right)^\top
\,
E(x)
\,
\left( z(t,\, x) - z^{(\upsigma)}(x)  \right)_t
\, dx.
\end{eqnarray*}
With (\ref{adjointsystem1611}) this yields
\begin{eqnarray*}
E_0'(t) & = &
\int_0^L
-\left( z(t,\, x) - z^{(\upsigma)}(x)  \right)^\top
\,
E(x)\,D(x)
\,
\left( z(t,\, x) - z^{(\upsigma)}(x)  \right)_x
\\
&
+
&
 \left( z(t,\, x) - z^{(\upsigma)}(x)  \right)^\top \, \left( |\eta_0| E(x) \, M(x)^\top - E(x) \,D'(x) \right)
\,\left( z(t,\, x) - z^{(\upsigma)}(x)  \right)
\,
dx.
\end{eqnarray*}

Integration by parts yields
\begin{eqnarray*}
E_0'(t) & = &
- \frac{1}{2}
\left( z(t,\, x) - z^{(\upsigma)}(x)  \right)^\top
\,
E(x)\,D(x)
\,
\left( z(t,\, x) - z^{(\upsigma)}(x)  \right)|_{x=0}^L
\\
& &
+
\int_0^L
\frac{1}{2}\,
\left( z(t,\, x) - z^{(\upsigma)}(x)  \right)^\top
\,
E'(x)\,D(x)
\,
\left( z(t,\, x) - z^{(\upsigma)}(x)  \right)\,dx
\\
& &
-
\int_0^L
\frac{1}{2}\,
\left( z(t,\, x) - z^{(\upsigma)}(x)  \right)^\top
\,
E(x)\,D'(x)
\,
\left( z(t,\, x) - z^{(\upsigma)}(x)  \right)\,dx
\\
& &
+
\int_0^L
 \left( z(t,\, x) - z^{(\upsigma)}(x)  \right)^\top \, |\eta_0| E(x) \, M(x)^\top
\,\left( z(t,\, x) - z^{(\upsigma)}(x)  \right)
\,
dx.
\end{eqnarray*}

 Due to the boundary conditions in
(\ref{adjointsystem1611}) the terms
that appear in $E_0'(t)$ and  only depend on  the boundary values vanish.
Define the symmetric matrix $M_0$ as
\begin{equation}
\label{M1definition}
M_0(x) = \frac{|\eta_0|}{2} \left[ E(x)\, M^\top(x)+ M(x) \, E(x) \right]  - \frac{1}{2} \, D'(x)\, E(x) .
\end{equation}
Then  we have
 \begin{eqnarray*}
  E_0'(t)
 & = &
\int_0^L
\left( z(t,\, x) - z^{(\upsigma)}(x)  \right)^\top
\,
\left( \frac{1}{2} E'(x)\,D(x)
+
M_0(x)
\right)
\,
\left( z(t,\, x) - z^{(\upsigma)}(x)  \right)
\, dx.
\end{eqnarray*}
%
 %
Choose $\mu_+<0$ and $\mu_-<0$
as in (\ref{corondeltb}).
Then we have
$E_0'(t) \geq \nu_0\, E_0(t)$.
For $H(t)$ defined as
$H(t) = E_0(T-t)
$
this yields
\begin{eqnarray*}
H'(t) & = & - E_0'(T-t) \leq - \nu_0 \, E_0(T-t)
= - \nu_0 \,H(t).
\end{eqnarray*}
Now Gronwall's inequality implies that $H(t)$ decays with the exponential rate $\nu_0$,
that is
for all $t\in (0,T)$ we have the inequality
$H(t) \leq H(0) \, \exp(-\nu_0 \, t)$. This implies
\begin{equation}
\label{02062017}
E_0(t)
=H(T-t)
\leq
  H(0) \, \exp(-\nu_0 (T-t) )
=
 E_0(T)
 \,
\exp(-\nu_0 (T - t))
.
\end{equation}
Note that
due to the terminal condition at the time $T$ in
(\ref{adjointsystem1611}), the number $E_0(T)$ is completely determined by the function
$z^{(\upsigma)}(x)$.
In fact, by the definition of $E_0$ in (\ref{lyapunovdefinition}) we have
\[E_0(T)
 =\frac{1}{2}
\int_0^L \exp(- \mu_+ \,x) \left( z_+^{(\upsigma)}(x)\right)^2
+
\exp( \mu_- \, \,x)  \left( z_-^{(\upsigma)}(x) \right)^2\, dx
.
\]
Hence $E_0(T)$ is independent of $T$.
The definition of $E_0(t)$ in (\ref{lyapunovdefinition})  implies that
\begin{equation*}
\label{lyapunovdefinition2017a}
\int_0^L \left( z(t,\, x) - z^{(\upsigma)}(x) \right)^\top
\,
\left( z(t,\, x) - z^{(\upsigma)}(x)  \right)
\, dx
\leq
\, 2 \exp(- \mu_-\, L)\, E_0(t).
\end{equation*}
With (\ref{02062017}) this implies
\begin{equation*}
\int_0^L \|z(t,\, x) - z^{(\upsigma)}(x) \|^2_{{\mathbb R}^2}\, dx
\leq
\, 2 \exp(- \mu_-\, L)\, E_0(T) \, \,
\exp(-\nu_0 (T - t)).
\end{equation*}
%
%
For  $t>0$, define $y_0^{(t)}(x)=z(t,\,x) - z^{(\upsigma)}(x)$.
Since $z^{(\upsigma)}(x)\in H^1(0,\,L)$,
the well-posedness result
 Theorem A.1 from \cite{BastinCoron2016}
applied to (\ref{adjointsystem1611})
implies that $y_0^{(t)}\in H^1(0,\,L)$.
  Theorem A.1 from \cite{BastinCoron2016}
  also yields the existence of a solution to the
following  initial boundary value problem
for $x\in (0,L)$ and $s \in [t,\,t+1]$:
\[
\left\{
\begin{array}{l}
z(t,\,x) - z^{(\upsigma)}(x) =  y_0^{(t)}(x),
\\
(z-z^{(\upsigma)})_t(s,x) + D\, (z-z^{(\upsigma)}) _x(s,x)  = -( \eta_0\,M(x)^\top  + D'(x))\, (z-z^{(\upsigma)})(s,x) ,
\\
z_+(s,\, L) - z_+^{(\upsigma)}(L)=  0,
\\
z_-(s,\,0) - z_-^{(\upsigma)}(0) =  0.
\end{array}
\right.
\]
%
Inequality (\ref{02062017a}) implies
\begin{eqnarray*}
& &
d_+(0)\,\|z_+(\cdot,\, 0) - z_+^{(\upsigma)}(0)\|_{L^2(t-1, t)}^2
+
|d_-(L)|\,\|z_-(\cdot,\,L) - z_-^{(\upsigma)}(L)\|_{L^2(t-1, t)}^2
\\
&
\leq
&
(1 + \tilde C_2)
\|y_0^{(t)}\|_{L^2((0,\,L); {\mathbb R}^2)}^2
\leq
 2 \, (1 + \tilde C_2) \,
  \exp(- \mu_-\, L)\, E_0(T) \, \exp(-\nu_0 (T - t))
.
\end{eqnarray*}
%
%
%
%
%
%
Define the constant
\[
{\hat C}
=
 2\, (1 + \tilde C_2) \, \exp(- \mu_-\, L)\, E_0(T)
\,\frac{\exp(\nu_0)}{1- \exp\left(\nu_0\right)}
.
\]
Then ${\hat C}$ is independent of $T$ and we have the inequality
\begin{eqnarray*}
& &d_+(0)\,\|z_+(\cdot,\, 0) - z_+^{(\upsigma)}(0)\|_{L^2(0,\, T)}^2
+
|d_-(L)|\, \|z_-(\cdot,\,L) - z_-^{(\upsigma)}(L)\|_{L^2(0, \, T)}^2
\\
&
\leq
&
\sum_{j\in
{\mathbb N}: j-1< T}
d_+(0)\,\|z_+(\cdot,\, 0) - z_+^{(\upsigma)}(0)\|_{L^2(j-1, j)}^2
+
|d_-(L)|\,\|z_-(\cdot,\,L) - z_-^{(\upsigma)}(L)\|_{L^2(j-1, j)}^2
\\
&
\leq
&
\sum_{j\in
{\mathbb N}:
T- j >-1}
 2 \, (1 + \tilde C_2) \,\exp(- \mu_-\, L)\, E_0(T) \, \exp(-\nu_0 (T - j))
\\
&
\leq
&
 2 \, (1 + \tilde C_2) \,\exp(- \mu_-\, L)\, E_0(T) \, \exp(\nu_0)
 \,\sum_{j=0}^\infty  \left(\exp(-\nu_0)\right)^j
  =
{\hat C}
.
\end{eqnarray*}

%
%
On account of (\ref{08062017a}), this implies
that there exists a constant $C_2>0$ that is independent of $T$ such that
for all $T>0$ we have the uniform bound
\begin{equation}
\label{vorletzte}
\left\|\left( F_T^\ast - F_{(\upsigma)}^\ast\right) p_0 \right\|_{H}\leq C_2.
\end{equation}
Now we consider
$
\left( {\cal M}_2
+
F_T^\ast  {\cal M}_4
\right)
\left((F_T - F_\upsigma) \, u^{(\upsigma)} \right)
 $.
First we show that
$
(F_T - F_\upsigma) \,  u^{(\upsigma)}
$
decays exponentially with time.
%
By definition, we have
\[   (F_T - F_\upsigma) \, u^{(\upsigma)}
=\left(
\begin{array}{r}
r_+(\cdot, L) - r_+^{(\upsigma)}(L)
\\
r_-(\cdot,\, 0) - r_-^{(\upsigma)}(0)
 \end{array}
\right),
\]
where for $x\in (0,L)$ and $t\in (0,T)$ we have
\begin{equation}
\label{linearizedsystem17112016}
\left\{
\begin{array}{l}
r(0,\,x) -r^{(\upsigma)}(x) = -r^{(\upsigma)}(x),
\\
(r - r^{(\upsigma)})_t + D\, (r - r^{(\upsigma)})_x = \eta_0\,M\, (r - r^{(\upsigma)}),
\\
r_+(t,\,0)- r^{(\upsigma)}_+(0) = 0,
\\
r_-(t,\, L) - r^{(\upsigma)}_-(L) = 0.
\end{array}
\right.
\end{equation}
Again similar to \cite{BastinCoron2016}
but this time for $\mu_+>0$ and $\mu_->0$
consider the Lyapunov function with exponential weights
\begin{equation}
\label{lyapunovdefinition17}
E_1(t)
 = \frac{1}{2}
\int_0^L {\rm e}^{- \mu_+ \,x} \left( r_+(t,\, x) - r_+^{(\upsigma)}(x) \right)^2
+
{\rm e}^{\, \mu_- \, \,x} \left( r_-(t,\, x) - r_-^{(\upsigma)}(x) \right)^2\, dx
\end{equation}
\[
 =
 \frac{1}{2}
\int_0^L  \left( r(t,\, x) - r^{(\upsigma)}(x) \right)^\top\, E(x)\,
 \left( r(t,\, x) - r^{(\upsigma)}(x) \right)\, dx
 .
\]
For the time derivative of $E_1$ we obtain with (\ref{linearizedsystem17112016})
\begin{eqnarray*}
E_1'(t) & = &
\int_0^L
-\left( r(t,\, x) - r^{(\upsigma)}(x)  \right)^\top
\,
E(x)\,D(x)
\,
\left( r(t,\, x) - r^{(\upsigma)}(x)  \right)_x
\\
& &
+\eta_0 \left( r(t,\, x) - r^{(\upsigma)}(x)  \right)^T \, E(x)\, M(x)
\,\left( r(t,\, x) - r^{(\upsigma)}(x)  \right)
\,
dx.
\end{eqnarray*}
Integration by parts yields
\begin{eqnarray*}
E_1'(t) & = &
- \frac{1}{2}
\left( r(t,\, x) - r^{(\upsigma)}(x)  \right)^\top
\,
E(x)\,D(x)
\,
\left( r(t,\, x) - r^{(\upsigma)}(x)  \right)|_{x=0}^L
\\
& &
+
\int_0^L
\frac{1}{2}\,
\left( r(t,\, x) - r^{(\upsigma)}(x)  \right)^\top
\,
E'(x)\,D(x)
\,
\left( r(t,\, x) - r^{(\upsigma)}(x)  \right)\,dx
\\
& &
-
\int_0^L
 \left( r(t,\, x) - r^{(\upsigma)}(x)  \right)^\top \, M_1(x)
\,\left( r(t,\, x) - r^{(\upsigma)}(x)  \right)
\,
dx,
\end{eqnarray*}
with the matrix $M_1$ as defined in (\ref{edefinition}).
 Due to the boundary conditions in
(\ref{linearizedsystem17112016}) the terms  coming from the boundary vanish.
 With  $\mu_+>0$,  $\mu_->0$, and $\nu_a<0$ as in (\ref{corondelta}),
 this yields the inequality
 \begin{eqnarray*}
  E_1'(t)
 & \leq &
\nu_a  \, E_1(t).
\end{eqnarray*}

Define $\nu_1 = - \nu_a>0$.
Then
Gronwall's inequality implies that $E_1(t)$ decays with an exponential rate $\nu_1$.
This implies that for all $t\in (0,T)$ we have
\[
E_1(t)
\leq
 E_1(0)
 \,
\exp(-\nu_1 \, t)
.
\]
For  $t>0$, define $z_0^{(t)}(x)=r(t,\,x) - r^{(\upsigma)}(x)$.
%
%
Since $r^{(\upsigma)}(x)\in H^1(0,\,L)$,
the well-posedness result
 Theorem A.1 from \cite{BastinCoron2016}
implies
$z_0^{(t)}\in H^1(0,\,L)$.
Now we apply  Theorem A.1 from \cite{BastinCoron2016} to the
 initial boundary value problem
for $x\in (0,L)$ and $s \in [0,\,1]$
\[
\left\{
\begin{array}{l}
r(t,\,x) - r^{(\upsigma)}(x) =  z_0^{(t)}(x),
\\
(r-r^{(\upsigma)})_t + D\, (r-r^{(\upsigma)}) _x =  \eta_0\,M\, (r-r^{(\upsigma)}),
\\
r_+(s,\, 0) - r_+^{(\upsigma)}(0)=  0,
\\
r_-(s,\,L) - r_-^{(\upsigma)}(L) =  0.
\end{array}
\right.
\]
Similar as in the discussion for $E_0$,
(\ref{08062017aa}) implies
\begin{eqnarray*}
& &
d_+(L)\,\|r_+(\cdot,\, L) - r_+^{(\upsigma)}(L)\|_{L^2(t,\, t+1)}^2
+
|d_-(0)|\,\|r_-(\cdot,\,0) - r_-^{(\upsigma)}(0)\|_{L^2(t, \,  t+1)}^2
\\
&
\leq
&
(1 + \tilde C_3)\,
\|z_0^{(t)}\|_{L^2((0,\,L); {\mathbb R}^2)}^2
\leq
 2 \, (1 + \tilde C_3)\,\exp(\mu_+\, L)\, E_1(0) \, \exp(-\nu_1 \, t)
.
\end{eqnarray*}
Hence we have the inequality
\begin{eqnarray*}
& &d_+(L)\,\|r_+(\cdot,\, L) - r_+^{(\upsigma)}(L)\|_{L^2(0,\, T)}^2
+
|d_-(0)|\, \|r_-(\cdot,\,0) - r_-^{(\upsigma)}(0)\|_{L^2(0, \, T)}^2
\\
&
\leq
&
\sum_{j\in
{\mathbb N}:
 j< T}
d_+(L)\,\|r_+(\cdot,\, L) - r_+^{(\upsigma)}(L)\|_{L^2(j, j+1)}^2
+
|d_-(0)|\,\|r_-(\cdot,\,L) - r_-^{(\upsigma)}(L)\|_{L^2(j, j+1)}^2
\\
&
\leq
&
\sum_{j\in
{\mathbb N}:
T- j >0}
 2 \, (1 + \tilde C_3)\,\exp( \mu_+\, L)\, E_1(0) \, \exp(-\nu_1 \,j)
\\
&
\leq
&
 2 \, (1 + \tilde C_3)\,\exp(\mu_+\, L)\, E_1(0)
 \,\sum_{j=0}^\infty  \left(\exp(-\nu_1)\right)^j
  =
\frac{
 2\, (1 + \tilde C_3)\, \exp(\mu_+\, L)}
 {1 - \exp(-\nu_1)}
 \, E_1(0)
.
\end{eqnarray*}
%
%
%
%
%
%
%
%
%
%
%
Note that
due to the initial condition in (\ref{linearizedsystem17112016})
 the number $E_1(0)$ is only determined by
$r^{(\upsigma)}$.
Hence there exists a  constant $C_4$  that is independent of $T$
%
%
%
such that  for all $T>0$,
we have
$\left\|\left( F_T - F_{(\upsigma)}\right) \, u^{(\upsigma)} \right\|_{\left(L^2(0,T)\right)^2}\leq C_4
$.
Due to (\ref{ftnormadjoint}) this implies
\begin{equation}
\label{letzte}
\left\|
\left( {\cal M}_2
+
F_T^\ast  {\cal M}_4
\right)
\,\left( F_T - F_{(\upsigma)}\right) \, u^{(\upsigma)} \right\|_{\left(L^2(0,T)\right)^2}\leq
C_4\,\left( \|{\cal M}_2\| + \|{\cal M}_4\|  C_N \right).
\end{equation}
Here $\|{\cal M}_2\|$  and  $\|{\cal M}_4\|$ are matrix norms for the Euclidean space ${\mathbb R}^2$.
%
Thus (\ref{c4assertion}) follows
from
(\ref{vorletzte}) and (\ref{letzte}) with the constant
$C_D =
C_2 +
C_4\, \left( \|{\cal M}_2\| + \|{\cal M}_4\|  C_N \right)
.
$
\end{proof}
Now we can prove  Theorem \ref{satz1}.

{\bf Proof of  Theorem \ref{satz1}:}
For the objective functional   of the dynamic optimal control problem
(\ref{ocplambda}) with the  representation as in (\ref{09122016a})
we introduce the notation
\[
\tilde J(u)
=
J(u,\, F_T u)
.
\]
Define
\begin{equation}
\label{qdefinition}
q(u - \bar u)
=
\tfrac{1}{2}
\,
\left\langle
\left(
(u - \bar u)_+, \,
(F_T (u - \bar u))_-
\right)^\top,
\,
A_0 \,
\left(
(u - \bar u)_+, \,
(F_T (u - \bar u) )_-
\right)^\top
\right\rangle_{  H }
\end{equation}
\[
+ \tfrac{1}{2}
\,
\left\langle
\left(
(u - \bar u)_-, \,
(F_T (u - \bar u))_+
\right)^\top,
\,
A_L \,
\left(
(u - \bar u)_-, \,
(F_T (u - \bar u))_+
\right)^\top
\right\rangle_{  H }
.
\]
For all  $u$ and $\bar u\in H$,
we can represent $\tilde J$ in the form
%
$
\tilde J(u )
$
\[
=
\tilde J(\bar u)
+
\langle
{\cal M}_1  \bar u + {\cal M}_2 F_T  \bar u + v_1  + F_T^\ast\left({\cal M}_3  \bar u + {\cal M}_4 F_T  \bar u + v_2\right),  u - \bar u
\rangle_{H}
%
%
%
%
+
q(u- \bar u )
.
%
\]
With the notation
$
D\tilde J(\bar u)
=
{\cal M}_1  \bar u + {\cal M}_2 F_T  \bar u + v_1  + F_T^\ast\left({\cal M}_3  \bar u + {\cal M}_4 F_T  \bar u + v_2\right)
$
we have
\begin{equation}
\label{tildeudefinitionentwicklunga}
\tilde J( u )=
\tilde J(\bar u)
+
\langle
D\tilde J(\bar u)
,
\,
u - \bar u
\rangle_H
+
q(u - \bar u ).
\end{equation}
For the optimal control
$u^{(\updelta,\,T)}$,
the necessary optimality condition
(\ref{optimalitysystem})
implies that $D\tilde J(\bar u^{(\updelta,\,T)})=0$,
hence we have
\begin{equation}
\label{tildeudefinitionentwicklunglambda}
\tilde J( u^{(\upsigma)} )=
\tilde J(u^{(\updelta,\,T)})
+
q(  u^{(\upsigma)} -  u^{(\updelta,\,T)} )
.
\end{equation}
On the other hand, we have
\begin{equation}
\label{tildeudefinitionentwicklungsigma}
\tilde J( u^{(\updelta,\,T)} )=
\tilde J( u^{(\upsigma)} )
+
\left\langle
D\tilde J( u^{(\upsigma)} )
,
\,
u^{(\updelta,\,T)} -  u^{(\upsigma)}
\right\rangle_H
+
q(  u^{(\updelta,\,T)} -   u^{(\upsigma)}  ).
\end{equation}
Adding up (\ref{tildeudefinitionentwicklunglambda}) and
(\ref{tildeudefinitionentwicklungsigma}) yields
\[
\tilde J( u^{(\upsigma)} ) + \tilde J( u^{(\updelta,\,T)} ) =
\tilde J( u^{(\upsigma)} ) + \tilde J( u^{(\updelta,\,T)} )
\]
\[
+
\left\langle
D\tilde J( u^{(\upsigma)} )
,
\,
u^{(\updelta,\,T)} -  u^{(\upsigma)}
\right\rangle_H
+
2\,
q(  u^{(\updelta,\,T)} -   u^{(\upsigma)}  ).
\]
This implies
\[
\left\langle
D\tilde J( u^{(\upsigma)} )
,
\,
u^{(\upsigma)} -  u^{(\updelta,\,T)}
\right\rangle_H
=
2\,
q(  u^{(\updelta,\,T)} -   u^{(\upsigma)}  )
.
\]
Since $A_0$ and $A_L$ are positive definite,
there exists a constant $\kappa>0$ such that
\begin{equation}
\label{qungleichung}
2 \,q(  u^{(\updelta,\,T)} -   u^{(\upsigma)}  ) \geq \kappa\, \| u^{(\updelta,\,T)} -   u^{(\upsigma)} \|_H^2.
\end{equation}
Hence we obtain the inequality
\[
\kappa\, \| u^{(\updelta,\,T)} -   u^{(\upsigma)} \|_H^2
\leq
\left\| D\tilde J( u^{(\upsigma)} ) \right\|_H
\;
\| u^{(\updelta,\,T)} -  u^{(\upsigma)}
\|_H.
\]
Thus we have
\begin{equation}
\label{abbruch}
\| u^{(\updelta,\,T)} -  u^{(\upsigma)}  \|_H
\leq
\frac{1}{\kappa}
 \, \| D\tilde J( u^{(\upsigma)} ) \|_H.
\end{equation}
In order to use (\ref{abbruch}) to prove (\ref{uniformbounded09122016}),
we need an upper bound for 
\begin{equation}
\label{09212016b}
\left\| D\tilde J( u^{(\upsigma)} ) \right\|_H
=
\left\|
{\cal M}_1  u^{(\upsigma)}  + {\cal M}_2 F_T  u^{(\upsigma)}  + v_1  + F_T^\ast\left({\cal M}_3  u^{(\upsigma)}  + {\cal M}_4 F_T  u^{(\upsigma)} + v_2\right)
\right\|_H.
\end{equation}
%
%
Due to inequality (\ref{abbruch})
and equation (\ref{09212016b}), with the
choice $\bar C = \frac{C_D}{\kappa}$,
inequality (\ref{c4assertion}) from
Lemma \ref{lemma2016}
implies (\ref{uniformbounded09122016}).
Then inequality
(\ref{19042018a})
from
Lemma \ref{exponentialdecayconservationlemma19042018}
yields (\ref{20042018l}).
Thus we have proved  Theorem \ref{satz1}.
\section{Analysis for the case with an integer-constraint} \label{derivationint}

\subsection{Turnpike structure for the problem with one-sided control}
\label{turnpikeonesided}
For $T>0$,
let
$u^{(\updelta,\,T)}\in H$
denote the optimal control
that solves
(\ref{ocplambdaswitch}).
In Theorem \ref{satzintegerturnpike}
we have assumed that
$\nu$ satisfies (\ref{16012017}) so
that no switching occurs in the optimal control.
Hence
for all $T>0$ the plus-component in the  optimal control  has the
form  $u_+^{(\updelta,\,T)} = \alpha \in {\mblue \cal F}$.
For the proof of
Theorem \ref{satzintegerturnpike}
we start with an auxiliary result
about the optimal control
problem where the $+$-component of the control is fixed
in advance.
\begin{lemma}
Let $T>0$ and $u^{(\updelta,\,T)}_+\in L^2(0,\, T)$ be given.
The control  $u^{(\updelta,\,T)}_-$ is a solution of the dynamic optimal control problem
\begin{equation}
\label{ocplambdaminus}
\left\{
\begin{array}{l}
\min_{u_-\in  L^2(0,\, T),\, (r_+(\cdot, L),\,r_-(\cdot,\,0)) \in H} J((u_+,u_-),\, (r_+(\cdot, L),\,r_-(\cdot,\,0))  )\;
\\
\mbox{\rm subject to
(\ref{linearizedsystem})}
\end{array}
\right.
\end{equation}
if and only if there exists a multiplier $p^{(\updelta,\,T)}$
such that the optimality system
\begin{equation}
\label{optimalitysystemminus}
\lambda\,u^{(\updelta,\,T)}_-
+(1-\lambda)
\,
\left(F_T^\ast \left(F_T\,  u^{(\updelta,\,T)}  - R^{\ourdesi} \right)\right)_-
=
0
\end{equation}
holds, that is for $(t,\, x) \in (0,\, T) \times (0,\, L)$ almost everywhere we have
(\ref{39a}) and
\begin{equation}
\label{39aminus}
\lambda\,u^{(\updelta,\,T)}_-(t) + (1-\lambda)
\,|d_-(L)|\,   p_-^{(\updelta,\,T)}(t,\, L) = 0.
\end{equation}
\end{lemma}
The proof is similar to the proof of
Lemma \ref{neccessaryoptimalityconditionsdynamic} and is therefore omitted.

Let  $u_+^{(\upsigma)}=\alpha \in {\mblue \cal F}$ be given.
Define the static optimal control problem
\begin{equation}
\label{ocplambdastaticswitchuplusgiven}
\left\{
\begin{array}{l}
\min_{
u^{(\upsigma)}_-\in {\mathbb R},
\,
 \,R^{(\upsigma)}\in ( L^{2}(0,\, L))^2} J_0(u^{(\upsigma)},\, R^{(\upsigma)})\;
\\
\mbox{\rm subject to
(\ref{staticlinearizedsystem})}.
\end{array}
\right.
\end{equation}
Let $u^{(\upsigma,\,\alpha)}$ denote
a  static optimal control that solves
(\ref{ocplambdastaticswitchuplusgiven}).
Thus in particular $u^{(\upsigma,\,\alpha)}_+=\alpha$.
Let  $R^{(\upsigma,\,\alpha)}$ denote the  state
generated by
$u^{(\upsigma,\,\alpha)}$
as a solution of (\ref{staticlinearizedsystem}).
Now we state the necessary optimality conditions
for  the static optimal control problem
(\ref{ocplambdastaticswitchuplusgiven}).
The number
$u^{(\upsigma,\, \alpha)}_-$ can only be a { static} optimal control
if there exists a multiplier $P^{(\upsigma,\,\alpha)}$
such that the optimality system
\begin{equation}
\label{optimalitysystemstaticminus}
 \lambda\,u^{(\upsigma,\,\alpha)}_-
+
(1 - \lambda)\,
\left(
F_{(\upsigma)}^\ast
\left(
F_{(\upsigma)} u^{(\upsigma,\,\alpha)}  - R^{\ourdesi}
\right)
\right)_-
=
0
\end{equation}
holds, i.e.,
if (\ref{39astatic14})
holds
with $R^{(\upsigma)}  = R^{(\upsigma,\,\alpha)}$,
$u^{(\upsigma)}  = u^{(\upsigma,\,\alpha)}$,
$P^{(\upsigma)}  = P^{(\upsigma,\,\alpha)}$
 and the equation
$ \lambda\, u^{(\upsigma,\,\alpha)}_- + (1 - \lambda)\, |d_-(L)|\,   P_-^{(\upsigma,\,\alpha)}(L) = 0$
is  satisfied.

In our analysis we use  the following lemma
that is similar to   Lemma \ref{lemma2016}:
\begin{lemma}
\label{lemma2016minus}
For $t\in [0,\, T]$ and $\alpha \in {\mblue \cal F}$
define the constant control
\begin{equation}
\label{110120171423}
 u^{(s,\, T)}_+(t) =
 \alpha,
 \;
 u^{(s,\, T)}_-(t) = u^{(\upsigma, \, \alpha)}_-
\end{equation}
where $u^{(\upsigma, \, \alpha)}_-$
is the solution of the static problem
(\ref{ocplambdastaticswitchuplusgiven}).
There exist  constants $C_4>0$, $C_E>0$ that are independent of $T$
such that   for all $T>0$  we have
\begin{equation}
\label{17022017a}
\left\|\left( F_T - F_{(\upsigma)}\right) \, (u^{(s,\, T)}) \right\|_{H}\leq C_4,
\end{equation}
\begin{equation}
\label{c4assertionminus}
\left\|
 \lambda\, u^{(s,\, T)}_-  + (1 - \lambda)\, \left( F_{T}^\ast \left(F_{T} \, u^{(s,\, T)}  - R^{\ourdesi} \right)\right)_- \right\|_{L^2(0,\,T)}
\leq
C_E.
\end{equation}
\end{lemma}
The proof is similar to the proof of  Lemma \ref{lemma2016} and is therefore omitted.

\subsection{The turnpike phenomenon with integer constraint}
\label{turnpikewithintegerconstraint}
In this section, we show that if the dynamic control
at the boundary point $x=0$ is fixed,
the corresponding optimal dynamic control at $x=L$   has a turnpike structure.
%
%
\begin{lemma}
\label{satz1minus}
For given $u_+^{(\updelta,\,T,\, \alpha)}(t)  = \alpha \in {\mblue \cal F}$,
let $u^{(\updelta,\,T,\, \alpha)}_-\in L^2(0,\,T)$ denote the optimal dynamic control
that solves (\ref{ocplambdaminus}).
There exists
a constant $\bar C>0$
that is independent of $T$ and $\alpha$
such that for all $T>0$ we have
\begin{equation}
\label{uniformbounded09122016minus}
\frac{1}{T}\;
\int_0^{T} \left|u^{(\updelta,\,T,\, \alpha)}_-(t) - u^{(\upsigma, \, \alpha)}_-\right|^2 \, d\,t
\leq \frac{1}{T} \;\bar C.
\end{equation}
\end{lemma}
\begin{proof}
%
Let $u_+(t)=\alpha\in {\mblue \cal F}$ be given.
For the objective functional   of the dynamic optimal control problem
(\ref{ocplambdaminus}) with the  representation as in (\ref{09122016a})
we introduce the notation
\[
\hat J(u_-)=
(1 - \lambda)\,
\|
F_T\, ((\alpha,\, u_-)) - R^{\ourdesi}
\|_{ H }^2
+
\lambda\,
\|(\alpha,\, u_-)\|_{ H }^2.
\]
Let $X_- = L^2(0,T)$.
For all  $u_-$ and $\bar u_-\in X_-$,
we can represent $\hat J$ in the form
\[
\hat J( u_-)=
\hat J( \bar u_-)
+
2\,
\langle
\lambda\, \bar u_- + (1 - \lambda)\,  \left( F_T^\ast\left(F_T(\alpha, \bar u_-) - R^{\ourdesi} \right)\right)_-,
\,\,
u_- - \bar u_-
\rangle_{X_-}
\]
\[
+(1 - \lambda)\,
\| \left(F_T ( (0, \, u_- - \bar u_- )) \right)_-\|^2_{H}  + \lambda\, \| u_- - \bar u_-\|^2_{X_-}
.
\]
With the notation
\[
D\hat J(\bar u_-)
=
2\, \left[\lambda\, \bar u_- + (1 - \lambda)\, \left(F_T^\ast\left(F_T(\alpha,\, \bar u_-) - R^{\ourdesi} \right)
\right)_- \right]
\]
we have
\[
\hat J( u_- )=
\hat J( \bar u_-)
+
\langle
D\hat J(\bar u_-)
,
\,
u_- - \bar u_-
\rangle_{X_-}
\]
\[
+
(1 - \lambda)\,\|\left(F_T ((0,\, u_- - \bar u_-) )\right)_-\|^2_{H}  +  \lambda\, \|u_- - \bar u_-\|^2_{X_-}.
\]
For the optimal control
$u_-^{(\updelta,\, T,\, \alpha)}$,
the necessary optimality condition
(\ref{optimalitysystemminus})
implies that $D\hat J(\bar u_-^{(\updelta,\,T,\, \alpha )})=0$.
Hence for  $u^{(s,\, T)}$
as defined in (\ref{110120171423})
 we have
\begin{equation}
\label{tildeudefinitionentwicklunglambdaminus}
\hat J( u^{(s,\, T)}_- )=
\hat J(u^{(\updelta,\, T,\,\alpha)}_-)
+
(1 - \lambda)\,\|F_T (0,\, u^{{(s,\, T)}}_-   - u^{(\updelta,\, T,\, \alpha)}_- )\|^2_{H}
 +  \lambda\, \|u^{(s,\, T)}_- -  u_-^{(\updelta,\, T,\,\alpha)}\|^2_{X_-}.
\end{equation}
%
As in the proof of  Theorem \ref{satz1}
using the necessary optimality condition
(\ref{optimalitysystemminus})
we obtain
\begin{equation}
\label{abbruchminus}
\| u^{(\updelta,\,T,\,\alpha)}_- -  u^{(s,\, T)}_-  \|_{X_-}
\leq
\frac{1}{2\, \lambda } \, \| D\hat J( u^{(s,\, T)}_- ) \|_{X_-}.
\end{equation}
In order to use (\ref{abbruchminus}) to prove (\ref{uniformbounded09122016minus}),
we need an upper bound for
\begin{equation}
\label{09212016bminus}
\frac{1}{2 } \,\| D\hat J( u^{(s,\, T)}_- ) \|_{X_-}
=
\left\|\lambda\,  u^{(s,\, T)}_- + (1 - \lambda)\,
\left[F_T^\ast\left(F_T (\alpha,\, u^{(s,\, T)}_-) - R^{\ourdesi} \right)\right]_- \right\|_{X_-}.
\end{equation}
Due to inequality (\ref{abbruchminus})
and equation (\ref{09212016bminus}),
 with the
choice $\bar C = \frac{C_E}{\lambda}$,
inequality (\ref{c4assertionminus}) from
Lemma \ref{lemma2016minus}
implies (\ref{uniformbounded09122016minus}).
Thus we have proved  Lemma \ref{satz1minus}.
\end{proof}
Now we prove Theorem \ref{satzintegerturnpike}.
For the optimal dynamic control at $x=0$,
due to the integer constraint (\ref{21122016}) we do not have optimality conditions,
so we have to use the fact that ${\cal F}$ is a finite set in the arguments.

{\bf Proof of  Theorem \ref{satzintegerturnpike}:}
If $\omega(T)=0$, the assertion follows immediately,
since both the solution of the dynamic and the solution of
the static problem are zero.

Assume that  $\omega(T)>0$.
We  show that if $T$ is sufficiently large,
the plus--component of
a solution of
the dynamic optimal control problem
 (\ref{ocplambdaswitch})
 also appears in a solution of
  the static problem
(\ref{ocplambdastaticswitch}).
For this purpose we consider the objective function.

For $\alpha \in {\cal F}$,
$u_-\in L^2(0,\,T)$ and $v\in  {\mathbb R}$
we introduce the notation
\[
\hat J_\alpha( u_- ) = (1 - \lambda)\,
\|
F_T\, ((\alpha,\, u_-)) - R^{\ourdesi}
\|_{ H}^2
+
\lambda\,
\|(\alpha,\, u_-)\|_{ H}^2,
\]
\[
\hat J_{0,\, \alpha}( v ) =
\lambda\,
\|(\alpha,\,v)^\top\|_{ {\mathbb R}^2}^2
+
(1 - \lambda)\,
\|
F_{(\upsigma)}\, (\alpha,\, v)^\top  - R^{\ourdesi}
\|_{{\mathbb R}^2  }^2
\]
and  $u^{(s,\, T,\, \alpha)}$ instead of
$u^{(s,\, T)}$
as defined in (\ref{110120171423})
in order to clarify the dependence on $\alpha \in {\cal F}$.
Then we can write
(\ref{17022017a})
in the form
\begin{equation}
\label{17022017b}
\left\|\left( F_T - F_{(\upsigma)}\right) \, u^{(s,\, T,\, \alpha)} \right\|_{H}\leq C_4
.
\end{equation}
%
We have
\begin{equation}
\label{09062017ar}
\omega(T) = \min_{\alpha \in {\cal F}} \hat J_\alpha( u_-^{(\updelta,\, T,\,\alpha)} )
\leq \min_{\alpha \in {\cal F}} \hat  J_\alpha( u_-^{(s,\, T,\, \alpha)} ).
\end{equation}
Inequality (\ref{uniformbounded09122016minus})
implies
$
\|u^{(\updelta,\, T,\, \alpha)} - u^{(s,\, T,\, \alpha)}\|_{H}\leq \sqrt{\bar C}
$.
%
Thus due to (\ref{tildeudefinitionentwicklunglambdaminus})  and
 (\ref{ftnorm}) we have
\begin{equation}
\label{16022017d}
\min_{\alpha \in {\cal F}}
\hat J_\alpha( u^{(s,\, T,\, \alpha)}_- )
\leq
\omega(T)
+
(1 - \lambda) \,
C_N^2 \, {\bar C} +  \lambda\, {\bar C}.
 \end{equation}
 Hence
 \[
 \omega(T)
 \geq
 \min_{\alpha \in {\cal F}}
\hat J_\alpha( u^{(s,\, T,\, \alpha)}_- )
-
(1 - \lambda) \,
C_N^2 \, {\bar C}
-  \lambda\, {\bar C}.
\]
%

Again $u^{(\upsigma,\,\alpha)}$ denotes
a  static optimal control that solves
(\ref{ocplambdastaticswitchuplusgiven}).
For all $\alpha \in {\cal F}$ we have
\begin{eqnarray}
\label{16022017a}
& & \hat J_\alpha( u^{(s,\, T,\, \alpha)}_- )
\\
\nonumber
& = &
T \,\lambda \, \|(\alpha,\, u_-^{(\upsigma,\,\alpha)} )^\top\|_{ {\mathbb R}^2}^2
+
(1 - \lambda) \, \|
F_T\, u^{(s,\, T,\, \alpha)} - R^{\ourdesi}
\|_{ H}^2
\\
\nonumber
& = &
T \,\lambda \, \|(\alpha,\, u_-^{(\upsigma,\,\alpha)} )^\top\|_{ {\mathbb R}^2}^2
+(1 - \lambda) \, \|
\left(F_T
-
 F_{(\upsigma)} + F_{(\upsigma)} \right)
\, u^{(s,\, T,\, \alpha)} - R^{\ourdesi}
\|_{  H }^2
\\
\nonumber
& \leq &
T \,\lambda \, \|(\alpha,\, u_-^{(\upsigma,\,\alpha)} )^\top\|_{ {\mathbb R}^2}^2
\\
\nonumber
& &
+(1 - \lambda) \,
\left(\|
\left(F_T
-
 F_{(\upsigma)} \right)
\, u^{(s,\, T,\, \alpha)}
\|_{ H }
+
\| F_{(\upsigma)}\, u^{(s,\, T,\, \alpha)} - R^{\ourdesi} \|_{ H }
\right)^2
\\
\nonumber
& \leq &
T\, \hat J_{0,\, \alpha}(u_-^{(\upsigma,\,\alpha)})
+(1 - \lambda)\, \left( C_4^2 +
2 \, C_4\, \| F_{(\upsigma)}\, u^{(s,\, T,\, \alpha)} - R^{\ourdesi} \|_{  H }
\right)
\\
\nonumber
& = &
T\, \hat J_{0,\, \alpha}(u_-^{(\upsigma,\,\alpha)} )
+(1 - \lambda)\, \left( C_4^2 +
2 \, C_4\, \sqrt{T} \, \| F_{(\upsigma)}\,u_-^{(\upsigma,\,\alpha)} - R^{\ourdesi} \|_{{\mathbb R}^2}
\right).
\end{eqnarray}
Define
\[
\upsilon =  \min_{\alpha \in {\cal F}} \hat  J_{0,\, \alpha}( u_-^{(\upsigma,\,\alpha)} ).
\]
Note that the number  $\upsilon$ is independent of $T$ and
equal to the optimal value of  the static problem (\ref{ocplambdastaticswitch}).
Due  (\ref{09062017ar}),
 (\ref{16022017a}) implies
\begin{equation}
\label{17022017c}
T \, \upsilon  \geq
\omega(T) -  (1 - \lambda)\, \left( C_4^2 + 2\,C_4\, \sqrt{T} \, \| F_{(\upsigma)}\,u_-^{(\upsigma,\,\alpha)} - R^{\ourdesi} \|_{{\mathbb R}^2}\right)
.
\end{equation}
Moreover we  have
\begin{eqnarray}
\label{16022017b}
& & \hat J_\alpha( u^{(s,\, T,\, \alpha)}_- )
\\
\nonumber
& = &
T \,\lambda \, \|(\alpha,\, u_-^{(\upsigma,\,\alpha)} )^\top\|_{ {\mathbb R}^2}^2
+(1 - \lambda) \, \|
\left(F_T
-
 F_{(\upsigma)} + F_{(\upsigma)} \right)
\, u^{(s,\, T,\, \alpha)} - R^{\ourdesi}
\|_{ H }^2
\\
\nonumber
& \geq &
T \,\lambda \, \|(\alpha,\, u_-^{(\upsigma,\,\alpha)} )^\top\|_{ {\mathbb R}^2}^2
\\
\nonumber
& &
+(1 - \lambda) \, \|
\left(\|
\left(F_T
-
 F_{(\upsigma)} \right)
\, u^{(s,\, T,\, \alpha)}
\|_{ H }
-
\| F_{(\upsigma)}\, u^{(s,\, T,\, \alpha)} - R^{\ourdesi} \|_{ H }
\right)^2
\\
\nonumber
& \geq &
T\, \hat J_{0,\, \alpha}( u_-^{(\upsigma,\,\alpha)} )
- (1 - \lambda)\,
2 \, C_4\, \| F_{(\upsigma)}\, u^{(s,\, T,\, \alpha)} - R^{\ourdesi} \|_{ H }
\\
\nonumber
& = &
T\, \hat J_{0,\, \alpha}( u_-^{(\upsigma,\,\alpha)} )
 - 2 \, (1 - \lambda)\,  C_4\, \sqrt{T} \, \| F_{(\upsigma)}\,u_-^{(\upsigma,\,\alpha)} - R^{\ourdesi} \|_{{\mathbb R}^2}
.
\end{eqnarray}

Choose  $\alpha \in \cal{F}$  that is {\em not} optimal for the
 static problem. Then
$
\hat J_{0,\, \alpha}(u_-^{(\upsigma,\,\alpha)} ) = \upsilon +  \epsilon(\alpha)
$
with
$ \epsilon(\alpha)=
\hat J_{0,\, \alpha}(u_-^{(\upsigma,\,\alpha)} )
-
 \upsilon > 0$.
Suppose that $u^{(\updelta,\, T,\,\alpha)}$ is
a solution of the dynamic problem  (\ref{ocplambdaswitch}).
Then due to (\ref{tildeudefinitionentwicklunglambdaminus}) we have
\[
\omega(T) =  \hat J_\alpha(u^{(\updelta,\, T,\,\alpha)}_-)
\]
\[
\geq
\hat J_\alpha( u^{(s,\, T,\, \alpha)}_- )
-
(1 - \lambda)\,\|F_T ( u^{{(s,\, T,\, \alpha)}}   - u^{(\updelta,\, T,\, \alpha)} )\|^2_{H}
 -  \lambda\, \|u^{(s,\, T,\, \alpha)} -  u^{(\updelta,\, T,\,\alpha)}\|^2_{H}
 \]
 \[
 \geq
 \hat J_\alpha( u^{(s,\, T,\, \alpha)}_- )
- ((1 - \lambda) \, C_N^2 \,  +  \lambda) \, {\bar C}
.
\]
Due to
(\ref{16022017b})  and
(\ref{17022017c})
this implies that
\begin{eqnarray*}
\omega(T)
&
 \geq
 &
 T\, \hat J_{0,\, \alpha}( u_-^{(\upsigma,\,\alpha)} )
 - 2 \, (1 - \lambda)\,  C_4\, \sqrt{T} \, \| F_{(\upsigma)}\,u_-^{(\upsigma,\,\alpha)} - R^{\ourdesi} \|_{{\mathbb R}^2}
  - ((1 - \lambda) \, C_N^2 \,  +  \lambda) \, {\bar C}
  \\
  & = &
 T ( \upsilon +  \epsilon(\alpha))
  - 2 \, (1 - \lambda)\,  C_4\, \sqrt{T} \, \| F_{(\upsigma)}\,u_-^{(\upsigma,\,\alpha)} - R^{\ourdesi} \|_{{\mathbb R}^2}
  - ((1 - \lambda) \, C_N^2 \,  +  \lambda) \, {\bar C}
  \\
&  \geq &
\omega(T) -  (1 - \lambda)\, \left( C_4^2 + 2\,C_4\, \sqrt{T} \, \| F_{(\upsigma)}\,u_-^{(\upsigma,\,\alpha)} - R^{\ourdesi} \|_{{\mathbb R}^2}\right)
 + T \,\epsilon(\alpha)
 \\
 & - &
 2 \, (1 - \lambda)\,  C_4\, \sqrt{T} \, \| F_{(\upsigma)}\,u_-^{(\upsigma,\,\alpha)} - R^{\ourdesi} \|_{{\mathbb R}^2}
  - ((1 - \lambda) \, C_N^2 \,  +  \lambda) \, {\bar C}.
\end{eqnarray*}
Since $\epsilon(\alpha)>0$, for sufficiently large $T$ we have
\[
T\, \epsilon(\alpha)
-
 4\,  \sqrt{T} \, (1 - \lambda)\,  C_4\, \| F_{(\upsigma)}\,u_-^{(\upsigma,\,\alpha)} - R^{\ourdesi} \|_{{\mathbb R}^2}
>
  (1 - \lambda)\, C_4^2
%
  + ((1 - \lambda) \, C_N^2 \,  +  \lambda) \, {\bar C}.
\]
But this yields $\omega(T) > \omega(T)$ which is a contradiction.
Hence if $T$ is sufficiently large,
$u^{(\updelta,\, T,\,\alpha)}$ cannot be a solution of the dynamic problem  (\ref{ocplambdaswitch}).
This implies that
for all solutions of
the dynamic optimal control problem with integer constraints (\ref{ocplambdaswitch}),
the plus-component $\alpha$ is such that
we have
$\hat J_{0,\, \alpha}(u_-^{(\upsigma,\,\alpha)} )
= \upsilon$,
that is $\alpha $  is the first component of a solution of the static problem
(\ref{ocplambdastaticswitch}).

Hence under the assumptions of
Theorem \ref{satzintegerturnpike},
$\alpha\in {\mblue \cal F}$
can be chosen
such that both the plus-component
of the solution of the dynamic optimal control problem
(\ref{ocplambdaswitch})
and the
plus-component
solution of the static optimal control problem
(\ref{ocplambdastaticswitchuplusgiven})
are equal to $\alpha$.
Therefore  Lemma \ref{satz1minus} implies
(\ref{uniformbounded09122016minushaupt}) for all $T>0$
for some $\alpha \in {\cal F}$
if $u^{(\upsigma)}=u^{(\upsigma,\, \alpha)}$ is chosen as
a  static optimal control that solves
(\ref{ocplambdastaticswitchuplusgiven}).
Thus we have proved Theorem \ref{satzintegerturnpike}.

\section{Application to gas pipeline operation}\label{applicationsection}
The motion of
gas in a long high-pressure pipeline can be modeled with the
one-dimensional
isothermal Euler equations
	\begin{equation}\label{eq:ISO1}
			\partial_t \varrho 		+\partial_x(\varrho v) 		
= 0,									\;\;
			\partial_t (\varrho v)	+\partial_x (p+\varrho v^2)	
= -\theta\varrho v |v| - g\varrho h',
	\end{equation}
where $\varrho$ denotes the density,
$v$ the velocity,
$p$ the pressure of the gas,
 $g$ the gravitational constant, $h'$ the slope of the pipe and $\theta$ is a friction coefficient. The mass flux per cross sectional area is then~$q = \varrho v$ in $\si{\flux}$.
An ideal gas 
yields a constant speed of sound $c = \sqrt{p / \varrho}$.
Since $p + \varrho v^2 = p(1+\frac{v^2}{c^2})$
for small velocities $|v| \ll c$ system \eqref{eq:ISO1}
can be approximated by
	\begin{equation}\label{eq:ISO2}
			\partial_t \varrho	+\partial_x q 			
= 0, \;\;
			\partial_t q 		+c^2\partial_x \varrho	
= -\theta {q |q|}/{\varrho} - gh' \varrho.
	\end{equation}
For further modeling details, we refer to \cite{bandahertyklar, gasserherty, guherty, gudil}.
The pipelines are usually operated near stationary states given by a constant flow $q\equiv\bar{q}$ and a density
distribution $\bar\rho$ in the pipe given by solution of the ordinary differential equation
	\begin{equation}\label{eq:ISO2-stationary}
		\begin{aligned}
			c^2\partial_x \varrho	&= -\theta{q |q|}/{\varrho} - gh' \varrho.
		\end{aligned}
	\end{equation}
A typical control problem for transmission system operators is to transfer the flow and pressure regime from one stationary state $(\bar{q}^0,\bar\rho^0)$ to a particular desired one
$(\bar{q}^T,\bar\rho^T)$ by choosing appropriate pressure and/or flow conditions at the entry and exit of the pipeline \cite{controllabilitynodalprofile}.
Prototypically, we consider the situation that $(\bar{q}^0,\bar\rho^0)$ is uniquely determined from \eqref{eq:ISO2-stationary} by known $q_L^0$ and $\rho_L^0$ at the exit $x=L$ and shall be transfered to $(\bar{q}^T,\bar\rho^T)$
that again is determined by \eqref{eq:ISO2-stationary} for certain desired $q_L^T$ and $\rho_L^T$ at $x=L$.
Therefore, we consider minimizing the following tracking type cost function
\begin{equation}\label{eq:gascost}
 \hat{J}=\int_0^T |\rho(t,L)-\rho_L^T|^2+\alpha\,|q(t,L)-q_L^T|^2\,dt+\lambda \int_0^T \rho(t,0)^2+\beta\,q(t,0)^2\,dt
\end{equation}
for some $\alpha,\beta,\lambda>0$ subject to a linearization of \eqref{eq:ISO2-stationary} at $(\bar{q}^0,\bar{\rho}^0)$. We show that this problem
can be analyzed with the techniques presented above and that the turnpike phenomenon obtained from Theorem~\ref{satz1} can also be verified numerically here.

In vector form $y=(\rho,q)^\top$, \eqref{eq:ISO2} can be written as
\begin{equation}\label{eq:ISO2vec}\partial_t y + A \partial_x y = G(y),\quad A=\begin{pmatrix}0 & 1\\c^2 & 0\end{pmatrix},~G(y)=\begin{pmatrix} 0 \\ -\theta\frac{q |q|}{\varrho} - gh' \varrho \end{pmatrix}. \end{equation}
The matrix $A$ has the eigenvalues $\lambda_1=c$ and $\lambda_2=-c$ with the corresponding left and right eigenvectors
$l_1=\begin{pmatrix} c & 1 \end{pmatrix}$, $l_2=\begin{pmatrix}-c & 1\end{pmatrix}$, $r_1=\begin{pmatrix}\frac{1}{c} & 1\end{pmatrix}^\top$, $r_2=\begin{pmatrix}-\frac{1}{c} & 1\end{pmatrix}^\top$.
Multiplying \eqref{eq:ISO2vec} by $l_1$ and $l_2$ yields a system in diagonal form
\begin{equation}\label{eq:ISO2diag}
R_t + D\, R_x = F(R),
 \;\; \mbox{\rm with }\quad  d_+(x)= c,\; d_-(x)=-c
\end{equation}
in variables $R=(R_+,R_-)^\top=(l_1 y,l_2 y)^\top=(c\rho +q,-c\rho+q)^\top$ with
\[F(R_+,R_-)=\begin{pmatrix}l_1 G(y) \\ l_2 G(y)\end{pmatrix}= - \left[\frac12 \theta c \frac{(R_+ + R_-)|R_+ + R_-|}{R_+ - R_-}+gh' \frac{R_+-R_-}{2c}\right] \begin{pmatrix}1 \\ 1\end{pmatrix}.\]
The original coordinates are obtained from $R$ using
\begin{equation}\label{eq:invtrafo}
\rho={(R_+-R_-)}/{2c},\quad q={(R_+ + R_-)}/{2}.\end{equation}
The linearization of \eqref{eq:ISO2}
at $(\bar{q}^0,\bar{\rho}^0)$ corresponds to a linearization of \eqref{eq:ISO2diag}
at $\bar R = (\bar R^+, \bar R^-)^\top=(c\bar\rho^0+\bar q^0,-c\bar \rho^0+\bar q^0)^\top$ and yields a system of the form
\begin{equation}\label{eq:ISO2diaglin}
  r_t + D  r_x = \eta M r
\end{equation}
with $\eta=-1$ and $M=-F'(\bar{R})$ in the variables $r=R-\bar{R}$.
Moreover, the linear transformations \eqref{eq:invtrafo} and $r=R-\bar{R}$ used in \eqref{eq:gascost} yield a
quadratic cost function
 of the type
 \eqref{objectiveneu}.
Theorem~\ref{satz1} therefore applies.

In order to verify this numerically, we discretized \eqref{eq:ISO2diaglin} using finite differences with a first order explicit in time and implicit in space upwind scheme with
$N_x=40$ discretization points in space, $N_t=816$ discretization points in time and the trapezoidal rule for the integration in \eqref{eq:gascost}. The same spatial discretization was used
for the corresponding stationary optimal control problem. The discretized problems were both implemented in GAMS \cite{GAMS2014} and solved using the interior point method
IPOPT \cite{IPOPT}. The parameters for a numerical example are listed in Table~\ref{tab:parameters}. The numerical results for this example are presented in Figure~\ref{fig:numresults}
and show that the dynamic optimal solution is very close to the stationary solution for about two thirds of the considered time horizon. This justifies using for example
feedback stabilization techniques to an optimal stationary state as a simple control principle as an alternative to solving a very difficult dynamical optimal control problem for
gas pipeline operation.


\begin{table}
\label{tab:parameters}
\caption{Choice of parameters for the numerical results in Figure~\ref{fig:numresults}}
{\small
\begin{center}
\begin{tabular}{llll}
\hline
Symbol & Explanation & Chosen Value & Unit\\
\hline
$\theta$ & friction factor & 0.05 & $\si{\friction}$\\
$g$ & gravitational constant & 9.81 & $\si{\gravity}$\\
$h'$ & pipe slope & 0.025 & ---\\
$c$ & speed of sound & 340 & $\si{\speed}$\\
$L$ & length of pipe & 10\,000 & $\si{\length}$\\
$T$ & length of time horizon & 600  & $\si{\time}$\\
$\rho^0_L$ & initial density at the exit& 35 & $\si{\density}$\\
$q^0_L$ & initial flux at the exit & 400 & $\si{\flux}$\\
$\rho^T_L$ & desired density at the exit & 40 & $\si{\density}$\\
$q^T_L$ & desired flux at the exit& 400 & $\si{\flux}$\\
$\alpha,\beta,\lambda$ & weighting factors in cost function & 0.01,0.01,0.1111 & ---\\
\hline
\end{tabular}
\end{center}
}
\end{table}

\begin{figure}
\caption{The computed optimal flow $q$ (left) and  optimal dynamic density $\rho$ (middle) and the relative error $e_\rho$ and $e_q$ compared to the optimal stationary state for the boundary trace at the entry (right)}
\begin{center}
\includegraphics[width=0.33\textwidth]{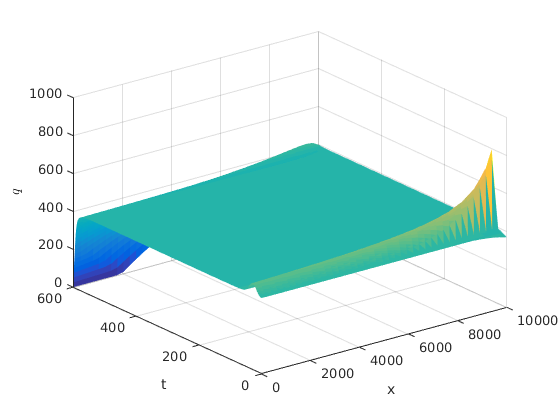}\hspace*{-0.15cm} \includegraphics[width=0.33\textwidth]{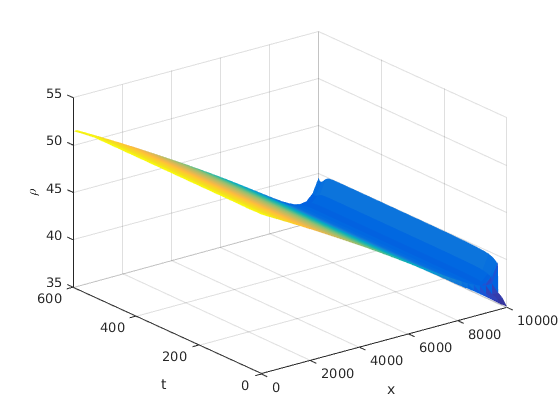} \hspace*{-0.15cm} \includegraphics[width=0.33\textwidth]{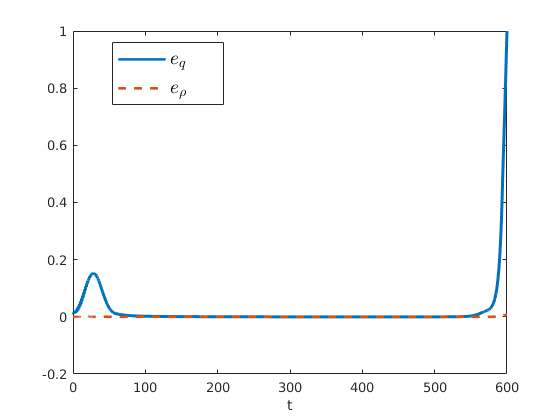}
\end{center}
\label{fig:numresults}
\end{figure}

\section{Conclusions}
\label{conclusions}

We have  shown that  controls that solve
optimal boundary control problems with linear hyperbolic systems
have a turnpike structure
in the sense that  the
$L^2$--norm of
the difference between the static optimal
control  and the dynamic optimal control
remains uniformly bounded for arbitrarily
long control times $T$.
Since the static optimal control is
constant with respect to time,
this means that
the dynamic optimal control
must
approach this
constant with increasing control time $T$
almost everywhere
on  the time interval $[0,\, T]$.
We have also given sufficient
conditions for the  turnpike phenomenon
for
optimal boundary control problems with
an additional integer constraint.
In this case the static problem is an
optimization problem with an integer constraint
and the turnpike phenomenon occurs
if
both
the switching cost
and the time interval $[0,\, T]$
are sufficiently
large.
It is not clear, if also for smaller
penalty parameters in the switching penalization
a turnpike phenomenon arises.
This is a question for future research.
%
%
%
%
Our results give  important insights
about the relation between the solutions
of the dynamic optimal boundary  control problems
and the corresponding static optimal control problem.
%
%
The results imply that for sufficiently large
control times, the static optimal controls
yield reasonable approximations
for the dynamic optimal boundary controls.


%

\medskip
\noindent{\bf Acknowledgment.}
This work is  supported by DFG
in the
Collaborative Research Centre
CRC/Transregio 154, Mathematical Modelling, Simulation and Optimization Using the Example of Gas Networks, project A03 and C03.

\bibliographystyle{siam}

\end{document}